\documentclass[a4paper]{article}

\usepackage{latexsym}
\usepackage{amsmath,amsthm,amsfonts,amssymb}
\usepackage{graphicx}

\newcommand{\bZ}{\mathbb{Z}}
\newcommand{\Z}{\mathbb{Z}}

\DeclareMathOperator{\Aut}{Aut}

\DeclareMathOperator{\Circ}{Circ}

\newcommand{\Sym}{{\rm Sym}}

\theoremstyle{plain}
\newtheorem{Theorem}{Theorem}
\newtheorem{Corollary}[Theorem]{Corollary}
\newtheorem{Proposition}[Theorem]{Proposition}

\newtheorem{Lemma}[Theorem]{Lemma}

\theoremstyle{remark}

\newtheorem*{Remarks}{Remarks}

\theoremstyle{definition}

\begin{document}

\title{$GI$-graphs  and their groups\\[+12pt] }

\author{
Marston D.E. Conder 
\\[+3pt]
{\normalsize Department of Mathematics, University of Auckland,}\\
{\normalsize Private Bag 92019, Auckland 1142, New Zealand} \\[+3pt]
{\normalsize m.conder@auckland.ac.nz}\\[+6pt]
\and
Toma\v{z} Pisanski 
\\[+3pt]
{\normalsize Faculty of Mathematics and Physics, University of Ljubljana,} \\
{\normalsize Jadranska 19, 1000 Ljubljana, Slovenia} \\[+3pt]
{\normalsize Tomaz.Pisanski@fmf.uni-lj.si}\\[+6pt] 
\and 
Arjana \v{Z}itnik
\\[+3pt]
{\normalsize Faculty of Mathematics and Physics, University of Ljubljana,} \\
{\normalsize Jadranska 19, 1000 Ljubljana, Slovenia} \\[+3pt]
{\normalsize Arjana.Zitnik@fmf.uni-lj.si} 
}


\date{10 July 2012}

\maketitle

\begin{abstract}
The class of generalized Petersen graphs was introduced by Coxeter in the 1950s.
Frucht, Graver and Watkins determined the automorphism groups of 
generalized Petersen graphs in 1971, and much later, Nedela and \v Skoviera 
and (independently) Lovre\v ci\v c-Sara\v zin characterised those which are  
Cayley graphs.
In this paper we extend the class of generalized Petersen graphs to a class 
of {\em $GI$-graphs}.  
For any positive integer $n$ and any sequence $j_0,j_1,....,j_{t-1}$ of integers mod $n$, 
the $GI$-graph $GI(n;j_0,j_1,....,j_{t-1})$ is a $(t\!+\!1)$-valent graph
on the vertex set $\bZ_t \times \bZ_n$, with edges of two kinds: 
\begin{itemize}
\item an edge from $(s,v)$ to $(s',v)$, for all distinct $s,s' \in \bZ_{t}$  
          and all $v \in \bZ_n$,
\item edges from $(s,v)$ to $(s,v + j_s)$ and  $(s,v - j_s)$, for all $s \in \bZ_{t}$ and $v \in \bZ_n$.
\end{itemize}

\noindent 
By classifying different kinds of automorphisms, 
we  describe the automorphism group of each $GI$-graph, 
and determine which $GI$-graphs are vertex-transitive and which are Cayley graphs.
A $GI$-graph can be edge-transitive only when $t \leq 3$, 
or equivalently, for valence at most $4$. 
We present a unit-distance drawing
of a remarkable $GI(7;1,2,3)$.
\medskip



\end{abstract}

\section{Introduction}
\label{intro}

Trivalent graphs (also known as cubic graphs) form an extensively studied class of graphs. 
Among them, the Petersen graph is one of the most important 
finite graphs, constructible in many ways, and is a minimal counter-example for many 
conjectures in graph theory. 
The Petersen graph is the initial member of a family of graphs $G(n,k)$,
known today as {\em Generalized Petersen graphs}, which have similar constructions. 
Generalized Petersen graphs were first introduced by Coxeter \cite{Coxeter} in 1950, 
and were named in 1969 by Watkins \cite{Watkins}. 

A standard visualization of a generalized Petersen graph consists of 
two types of vertices: half of them belong to an outer rim, and the other half belong 
to an inner rim; and there are three types of edges: those in the outer rim, those in 
the inner rim, and the `spokes', which form a $1$-factor between the inner rim and the outer rim. 
The outer rim is always a cycle, while the inner rim may consist of several isomorphic cycles. 
A generalized Petersen graph $G(n,k)$ is given by two parameters $n$ and $k$, 
where $n$ is the number of vertices in each rim, and $k$ is the `span' of the inner rim 
(which is the distance on the outer rim between the neighbours of two adjacent vertices 
on the inner rim). 

The family $G(n,k)$ contains some very important graphs. 
Among others of particular interest are the $n$-prism $G(n,1)$, the D\"{u}rer graph $G(6,2)$, 
the M\"{o}bius-Kantor graph $G(8,3)$, the dodecahedron $G(10,2)$, 
the Desargues graph $G(10,3)$, the Nauru graph $G(12,5)$, and of course the Petersen graph 
itself, which is $G(5,2)$.

Generalized Petersen graphs possess a number of interesting properties.
For example, $G(n,k)$ is vertex-transitive if and only if either 
$n = 10$ and $k = 2$, or $k^2 \equiv \pm 1$ mod $n\,$ \cite{Frucht},  
and a Cayley graph if and only if $k^2 \equiv 1$ mod $n\,$ \cite{Lovrecic1, NedelaSkoviera}, 
and arc-transitive only in the following seven cases:
$(n,k)  = (4,1)$, $(5,2)$, $(8,3)$, $(10, 2)$, $(10, 3)$, $(12, 5)$ or  $(24, 5)\,$ \cite{Frucht}. 

If we want to maintain the symmetry between the two rims, then another parameter has to 
be introduced, allowing the span on the outer rim to be different from 1. This gives the 
definition of an {\em $I$-graph}.

The family of $I$-graphs  was introduced in 1988 in the Foster Census \cite{Foster}. 
For some time this family failed to attract the attention of many researchers, 
possibly due to the fact that among all $I$-graphs, the only ones that are vertex-transitive 
are the generalized Petersen graphs \cite{igraphs,LovrecicMarusic}. 
Still,  necessary and sufficient conditions for testing whether or not two $I$-graphs are 
isomorphic were determined in \cite{igraphs,HPZ2}, and these were used to enumerate 
all $I$-graphs in \cite{Petkovsek}.
Also in \cite{HPZ2} it was shown that all generalized Petersen graphs are 
unit-distance graphs, by representing them as isomorphic $I$-graphs.
Furthermore, in \cite{igraphs} it was shown that automorphism group of a connected 
$I$-graph $I(n,j,k)$ that  is not a generalized Petersen graph is either dihedral or 
a group with presentation
\begin{equation*}                                
  \Gamma = \langle\, \rho, \tau, \varphi \mid
    \rho^n = \tau^2 = \varphi^2 = 1, \, 
    \rho\tau\rho = \tau, \, 
    \varphi\tau\varphi = \tau, \,
    \varphi\rho\varphi = \rho^a \, \rangle 
\end{equation*}
for some $a \in \bZ_n$, 
and that among all $I$-graphs, only the 
generalized Petersen graphs can be vertex-transitive or edge-transitive.

In this paper we further
generalize both of these families of graphs, and call them
\emph{generalized I-graphs}, or simply {\em $GI$-graphs}. 
%
We determine the group of automorphisms of any $GI$-graph. Moreover,
we completely characterize the edge-transitive, vertex-transitive and Cayley graphs,
among the class of $GI$-graphs.

At the end of the paper we briefly discuss the problem of unit-distance realizations of $GI$-graphs.
This problem has been solved for $I$-graphs in \cite{HPZ1}. 
We found a remarkable new example of a 4-valent unit-distance graph, 
namely $GI(7;1,2,3)$, which is a Cayley graph on 21 vertices for the group $\bZ_7 \rtimes \bZ_3$. 

Let us note that ours is not the only possible generalization.
For instance, see \cite{Lovrecic} for another approach, 
which is not much different from ours.
The basic difference is that our approach 
uses complete graphs, while the approach by Lovre\v ci\v c-Sara\v zin, 
Pacco and Previtali in \cite{Lovrecic} uses cycles; 
their construction coincides with ours for $t \le 3$, but not for larger $t$.  

We acknowledge the use of {\sc Magma} \cite{Magma} in constructing and analysing 
examples of $GI$-graphs, and helping us to see patterns and test conjectures 
that led to many of the observations made and proved in this paper.

\section{Definition of $GI$-graphs and their properties}
\label{defn-props}

For positive integers $n$ and $t$ with $n \ge 3$, 
let $(j_0,j_1, \dots , j_{t-1})$ be any sequence of integers 
such that $0 <  j_k < n$ and $j_k \ne n/2$, for $0 \le k < t$. \\[-8pt] 

Then we define $GI(n;j_0,j_1, \dots , j_{t-1})$ to be the graph 
with vertex set $\bZ_t \times \bZ_n$, and with edges of two types: 
\\[+4pt] 
\begin{tabular}{ll}
(a)\hskip -7pt ${}$ & an edge from $(s,v)$ to $(s',v)$, for all distinct $s,s' \in \bZ_{t}$ and all $v \in \bZ_n$, \\[+2pt] 
(b)\hskip -7pt ${}$ & edges from $(s,v)$ to $(s,v + j_s)$ and  $(s,v - j_s)$, for all $s\in \bZ_{t}$ and all $v \in \bZ_n$. \\[+4pt] 
\end{tabular}

This definition gives us an infinite family of graphs, which we call  \emph{GI-graphs}.

The graph $GI(n;j_0,j_1, \dots , j_{t-1})$ has $nt$ vertices, and is regular of valence $(t-1)+2 = t+1.$
Edges of type (a) are called the \emph{spoke edges}, 
while those of type (b) are called the \emph{layer edges}.
%
Also for each $s \in \bZ_t$ the set $L_s = \{(s,v) : v \in \bZ_n\}$ 
is called a {\em layer}, and 
for each $v \in \bZ_n$ the set $S_v = \{(s,v) : s \in \bZ_t\}$ 
is called a {\em spoke}.
We observe that the induced subgraph 
on each spoke is a complete graph $K_t$ of order $t$.  
On the other hand, the induced subgraph on the layer $L_s$ is a union of $d$ cycles 
of  length $n/d_s$, where $d_s =\gcd(n,j_s)$. \\[-8pt] 

In the case  $t=1$, the graph $GI(n;j_0)$ is simply a union 
of disjoint isomorphic cycles of length $n/\gcd(n,j_0)$. 
In the case $ t = 2$, we have $I$-graphs; for example, 
$GI(n;1,j)$ is a generalized Petersen graph, for every $j$, 
and in particular, $GI(5;1,2)$ is the dodecahedral graph 
(the $1$-skeleton of a dodecahedron). 
Some other examples are illustrated in Figure~\ref{fig:examplesGIgraphs}.

\begin{figure}[htb]
\centering
\begin{minipage}{0.3\textwidth}
  \centering
  \includegraphics[width=0.9\textwidth]{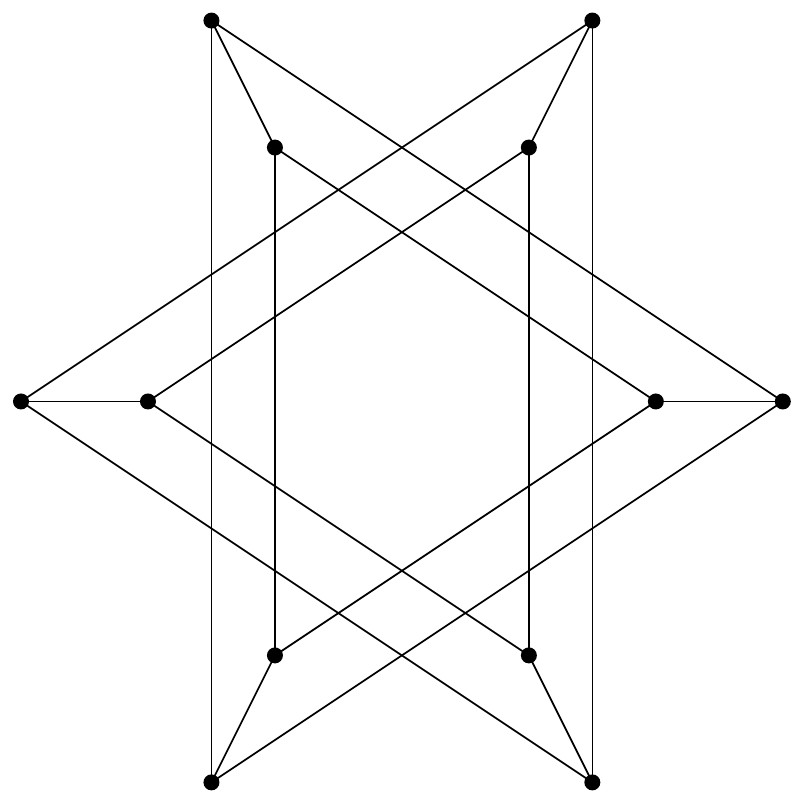}\\ (a)
\end{minipage}
\hfill
\begin{minipage}{0.3\textwidth}
  \centering
  \includegraphics[width=0.9\textwidth]{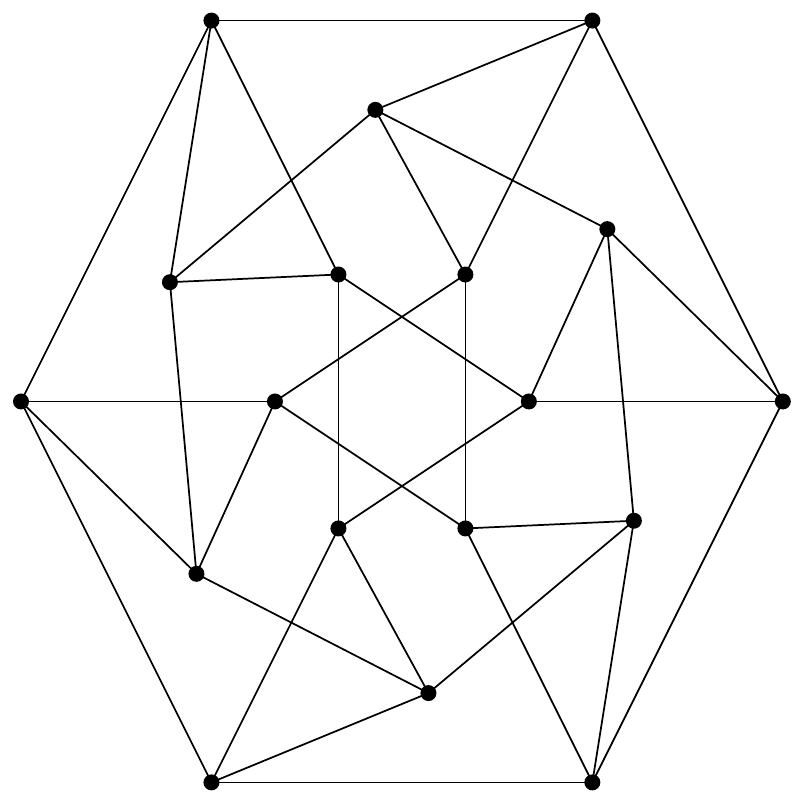}\\ (b)
\end{minipage}
\hfill
\begin{minipage}{0.3\textwidth}
  \centering
  \includegraphics[width=0.9\textwidth]{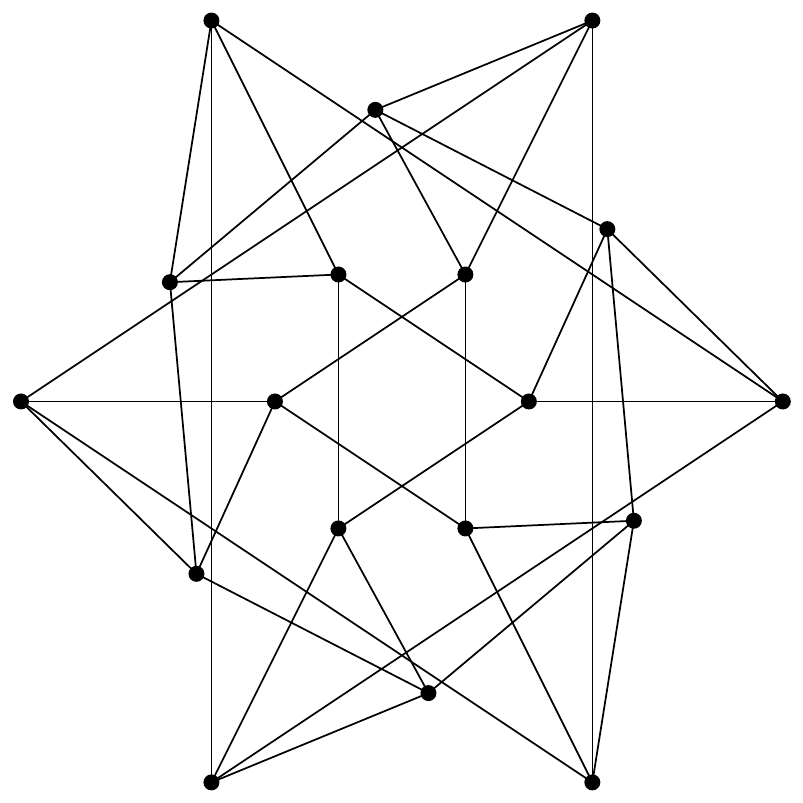}\\ (c)
\end{minipage}
\caption{$GI$-graphs $GI(6;2,2)$, $GI(6;1,1,2)$, and $GI(6;2,1,2)$.} 
\label{fig:examplesGIgraphs}
\end{figure}

Note that taking  $j_k^\prime = \pm j_k$ for all $k$ gives a $GI$-graph 
$GI(n;j_0^\prime,j_1^\prime, \dots , j_{t-1}^\prime)$
that is exactly the same as 
$GI(n;j_0,j_1, \dots , j_{t-1})$. 
Similarly, any permutation of
$j_0, j_1, \dots, j_{t-1}$ gives a $GI$-graph isomorphic to 
$GI(n;j_0,j_1, \dots , j_{t-1})$.
Therefore we will usually assume that $0 < j_k < n/2$ for all $k$, 
and that $j_0 \le j_1 \le \dots \le j_{t-1}$. 
In this case, we say that the $GI$-graph
$GI(n;j_0,j_1, \dots , j_{t-1})$ is in \emph{standard form}. 

\smallskip
The following theorem gives a partial answer to the problem of 
distinguishing between two $GI$-graphs. 

\begin{Proposition}  
\label{thm:multiply}
%
Suppose $j_0,j_1,\dots,j_{t-1} \ne 0$ or $n/2 $ modulo $n$, and 
$a$ is a unit in $\bZ_n$. 
Then the  graph $GI(n;aj_0, aj_1, \dots, aj_{t-1})$ is isomorphic to 
the graph $GI(n;j_0,j_1, \dots, j_{t-1})$.
\end{Proposition}
\begin{proof}
Since $a$ is coprime to $n$, 
the numbers $av$ for $0 \le v < n$, are all distinct in $\bZ_n$, 
and so we can label the vertices of $GI(n;aj_0, aj_1, \dots, aj_{t-1})$ 
as ordered pairs $(s,av)$ for $s \in \bZ_t$ and $v \in \bZ_n$. 
Now define a mapping
$\varphi : V(GI(n;aj_0,aj_1, \dots, aj_{t-1})) \to V(GI(n;j_0,j_1, \dots, j_{t-1}))$ 
by setting $ \varphi((s,av) ) = (s,v)$  for all $s \in \bZ_t$ and all $v \in \bZ_n$. 
This is clearly a bijection, and since a vertex $(s,av)$ in $GI(n;aj_0, aj_1, \dots, aj_{t-1})$ 
is adjacent to $(s',av)$ for each $s' \in \bZ_t \setminus \{s\}$ 
and to $(s,av \pm aj_s) = (s,a(v \pm j_s))$, it is easy to see that 
$\varphi$ is also a graph homomorphism.
\end{proof}
\medskip

We may collect the parameters $j_s$ into a multiset $J$, and then use the abbreviation 
$GI(n; J)$ for the graph $GI(n;j_0,j_1, \dots, j_{t-1})$. 
Also we will say that the multiset $J$ is in \emph{canonical form} if it is lexicographically first
among all the multisets that give isomorphic copies of $GI(n; J)$ via Proposition~\ref{thm:multiply}. 

We now list some other properties of $GI$-graphs. 
Proposition~\ref{thm:spokes} shows that the spoke edges are easy to recognise when $t > 3$. 

\begin{Proposition}
\label{thm:factors}
The graph $GI(n;j_0,j_1, \dots , j_{t-1})$ admits a factorization into 
a $(t-1)$-factor $nK_t$ and a $2$-factor $($namely the spokes and the layers$)$.
\end{Proposition}

\begin{Proposition} 
\label{thm:spokes}
An edge of a $GI$-graph with $4$ or more layers is a spoke-edge 
if and only if it belongs to some clique of size $4$.
\end{Proposition}
\begin{proof}
No edge between two vertices in the same layer can lie in a $K_4$ subgraph,  
because the subgraph induced on each layer is a union of cycles, 
and no two spokes between two different layers can have a common vertex. 
\end{proof} 


\begin{Proposition}
\label{thm:connectedGI}
Let $d=\gcd(n,j_0,j_1,\dots, j_{t-1})$. Then the graph
$GI(n;j_0,j_1, \dots , j_{t-1})$ is a disjoint union of $d$ copies of
$GI(n/d;j_0/d,j_1/d, \dots , j_{t-1}/d)$. In particular, the 
graph $GI(n;j_0,j_1, \dots , j_{t-1})$ is connected if and only if $d = 1$. 
\end{Proposition}
\begin{proof}
First observe that the edges of every spoke make up a clique (of order $t$), 
so the graph is connected if and only if every two spokes are connected via 
the layer edges.  Now there exists an edge between two spokes $S_u$ and $S_v$ 
whenever $v-u$ is a multiple of $j_s$ for some $s$, and hence a path of 
length $2$ between $S_u$ and $S_v$ whenever $v-u$ is a $\bZ_n$-linear 
combination of some $j_s$ and $j_{s'}$, and so on.  Thus $S_u$ and $S_v$ lie 
in the same connected component  of the graph if and only if $v-u$ is expressible 
(mod $n$) as a $\bZ$-linear combination of $j_0,j_1,\dots, j_{t-1}$, 
say $v-u = cn + c_{0}j_{0} + c_{1}j_{1} + \dots + c_{t-1} j_{t-1}$ 
for some  $c_0,c_1,\dots, c_{t-1} \in \bZ$. By Bezout's identity, this occurs if and 
only if $v-u$ is a multiple of $\gcd(n,j_0,j_1,\dots, j_{t-1}) = d$.  
It follows that the graph has $d$ components, each containing a set of spokes 
$S_v$ with $v = u+jd$ for fixed $u$ and variable $j$, that, is, with subscripts 
differing by multiples of $d$.  Finally, since 
$(v-u)/d = c(n/d)+ c_{0}(j_{0}/d) + c_{1}(j_{1}/d) + \dots + c_{t-1}(j_{t-1}/d)$, 
it is easy to see that each component is isomorphic to $GI(n/d;j_0/d,j_1/d, \dots , j_{t-1}/d)$.
\end{proof}
\medskip

Finally, note that the restriction of a $GI$-graph to any proper subset of its layers 
gives rise to another $GI$-graph. 
In particular, if $J$ and $K$ are multisets with $J \subseteq K$, 
then $GI(n;J)$ is an induced subgraph of $GI(n;K)$.

\section{Automorphisms of $GI$-graphs}
\label{automs}

In this section, we consider the possible automorphims of a $GI$-graph 
$X=GI(n;J)$, where $J =\{j_0,j_1, \dots , j_{t-1}\}$ is any multiset.
If $X$ is disconnected, then since all connected components of $X$ 
are isomorphic to each other (by Proposition \ref{thm:connectedGI}), 
we may simply reduce this to the consideration of automorphisms 
a connected component of $X$ (and then find the automorphism group 
using a theorem of Frucht \cite{Frucht0}, cf. \cite{Harary}).
Hence from now on, we will assume that $X$ is connected.

The set of edges of $X=GI(n;J)$ may be partitioned into spoke edges and layer edges, 
and we will call this partition of edges the \emph{fundamental edge-partition} of $X$.
We know that the graph induced on the spoke edges is a collection of complete graphs, 
and that the graph induced on the layer edges is a collection of cycles 
(with each cycle belonging to a single layer, but with a layer being composed 
of two or more cycles of the same length $n/\gcd(n,j_s)$ if the corresponding 
element $j_s$ of $J$ is not a unit mod $n$). 

We will say that an automorphism of $X$ \emph{respects the fundamental
edge-partition} if it takes spoke edges to spoke edges, and layer
edges to layer edges. 
Any automorphism of $X$ that does not respect the fundamental edge-partition 
(and so takes some layer edge to a spoke edge, and some spoke edge to a layer edge) 
will be called \emph{skew}. 

\begin{Theorem}
\label{theorem:skewautom}
Let $X$ be a connected $GI$-graph with $t$ layers, where $t \ge 2$. 
If $X$ has a skew automorphism, then either $\,t = 2$ and $X$ is isomorphic to one
of the seven special generalized Petersen graphs
$G(4,1)$, $G(5,2)$, $G(8,3)$, $G(10, 2)$, $G(10, 3)$, $G(12, 5)$ and $G(24, 5)$, 
or  $\,t = 3$ and $X$ is isomorphic to $GI(3;1,1,1)$. 
Moreover, each of these eight graphs is  arc-transitive 
$($and is therefore both vertex-transitive and edge-transitive$)$. 
\end{Theorem}
\begin{proof}
First, if $t > 3$ then no layer edge lies in a clique of size $t$, but every spoke edge does, 
and therefore no automorphism can map a spoke edge to a layer edge. Thus $t \le 3$. 

Next, suppose $t = 3$, and let $\varphi$ be an automorphism taking an edge $e$ of some 
spoke $S_v$ to an edge $e'$ of some layer $L_s$.  
Since every edge of a spoke lies in a triangle, namely the spoke itself, 
it follows that $\varphi$ must take the whole spoke $S_v = \{(0,v),(1,v),(2,v)\}$ containing 
$e$ to some triangle containing the layer edge $e'$, and then the other two edges of the 
triangle $\{\varphi(0,v),\varphi(1,v),\varphi(2,v)\}$ must be edges from the same layer as $e'$, 
namely $L_s$.  It follows that $j_s = n/3$.  But then since each of the images 
$\varphi(0,v),\varphi(1,v),\varphi(2,v)$ lies in two triangles (namely a spoke and 
a triangle in $L_s$), each of the vertices $(0,v),(1,v)$ and $(2,v)$ must similarly lie in 
two triangles, and it follows that all three layers contain a triangle, so $j_0 = j_1 = j_2 = n/3$. 
In particular, $\gcd(j_0,j_1,j_2) = n/3$, and by connectedness, Proposition~\ref{thm:connectedGI} 
implies $n/3 = 1$, so $n = 3$ and $j_0 = j_1 = j_2 = 1$. 
Thus $X$ is $GI(3;1,1,1)$, which is well-known to be arc-transitive 
(see \cite{Lovrecic}, for example).

Finally, for the case $t =2$, everything we need was proved in \cite{Frucht} and \cite{igraphs}.
\end{proof}

\begin{Corollary}
\label{corollary:edgetransitive}
Every edge-transitive connected $GI$-graph is isomorphic to one of the eight graphs 
listed in Theorem~{\em\ref{theorem:skewautom}}. 
\end{Corollary}

Hence from now on, we will consider only the automorphisms that respect the 
fundamental edge-partition.  
There are three special classes of such automorphisms:
 \\[+4pt] 
\begin{tabular}{cl} 
(1) & ${}$\hskip -8pt automorphisms that preserve every layer \\[+2pt] 
(2) & ${}$\hskip -8pt automorphisms that preserve every spoke \\[+2pt] 
(3) & ${}$\hskip -8pt automorphisms that permute both the layers and the spokes non-trivially.\\[+6pt] 
\end{tabular}
\\ 
\noindent 
We will consider particular cases of automorphisms of these types below. 

\medskip
Define mappings $\rho:V(X) \to V(X)$ and $\tau: V(X) \to V(X)$ given by 
\begin{equation*}   
  \rho(s,v) = (s,v+1) 
\ \ \
\mbox{and}
\ \ \ 
  \tau(s,v) = (s,-v) \quad  \hbox{ for all } s \in \bZ_t \hbox{ and all } v \in \bZ_n. \tag{$\dagger$}
\end{equation*}
Clearly these are automorphisms of $X$ of type (1), permuting the vertices in each
layer.  Indeed $\rho$ can be viewed as a rotation (of order $n$), and $\tau$ as a 
reflection (of order $2$), and it follows that the automorphism group of  $X$ 
contains a dihedral subgroup of order $2n$, generated by $\rho$ and $\tau$. 
These $2n$ automorphisms are all of type (1), and all of them respect the fundamental
edge-partition of $X$.

Next, if two of the members of the multiset $J$ are equal, 
say $j_{s_1}=j_{s_2}$ for $s_1 \ne s_2$,  then we have an automorphism 
$\lambda_{i,s_1,s_2}$ that  exchanges two cycles of layers 
$L_{s_1}$ and $L_{s_2}$, but preserves every spoke. These automorphisms are  of type (2).

\begin{Proposition}  
\label{propn:mix_layers}
Suppose $j_{s_1}=j_{s_2}$ where $s_1 \ne s_2$, and define $d=\gcd(n,j_{s_1}) = \gcd(n,j_{s_2})$. 
Then for each $i \in \bZ_d$, the mapping $\lambda_{i,s_1,s_2}:V(X) \to V(X)$  given  by
$$\lambda_{i,s_1,s_2}(s,v)=\left\{ \begin{array}{ll}
      (s_2,v) & \mbox{if} \ \  s=s_1 \ \ \mbox{and} \ \ v \equiv i \ \ \mbox{\rm mod} \ d \\
      (s_1,v) & \mbox{if} \ \  s=s_2 \ \ \mbox{and} \ \ v \equiv i \ \ \mbox{\rm mod} \ d \\
      \,(s,v) &  \mbox{otherwise}
       \end{array} \right.
$$
is an automorphism of $X$, which respects the fundamental edge-partition, 
and preserves all layers other than $L_{s_1}$ and $L_{s_2}$.       
\end{Proposition}
\begin{proof} 
This is obviously a permutation of $V(X)$, preserving adjacency. 
Moreover, it is also clear that $\lambda_{i,s_1,s_2}$ preserves every spoke $S_v$, 
and exchanges one of the cycles in layer $L_{s_1}$ with the corresponding 
cycle in layer $L_{s_2}$, while preserving all other layer cycles. 
\end{proof}

\begin{Corollary}
\label{cor:exchange_layers}
Suppose $j_{s_1}=j_{s_2}$ where $s_1 \ne s_2$, and define $d=\gcd(n,j_{s_1}) = \gcd(n,j_{s_2})$. \\[+1pt] 
Then the product   
$$\lambda_{s_1,s_2} := \lambda_{0,s_1,s_2}\lambda_{1,s_1,s_2}\dots \lambda_{d-1,s_1,s_2}$$
is an automorphism of $X$ that respects the fundamental edge-partition, 
and exchanges layers $L_{s_1}$ and $L_{s_2}$, while preserving every other layer.       
\end{Corollary}

There is another family of automorphisms exchanging layers that exist in 
some situations; but these automorphism do not preserve spokes,
and so they are of type (3): 

\begin{Proposition}
\label{propn:change_layers}
Let $a$ be any unit in $\bZ_n$ with the property that $aJ =\{\pm j_0, \pm j_1, \dots, \pm j_{t-1}\}$, 
and then let $\alpha: \bZ_t \to \bZ_t$ be any bijection with the property that
$j_{\alpha(s)} = \pm a j_s$ for all $s \in \bZ_t$.\\[+1pt] 
Then the mapping $\sigma_a: V(X) \to V(X)$ given by
$$
\sigma_a(s,v) = (\alpha(s),a v) \ \ \hbox{ for all } s \in \bZ_t \, \hbox{ and all } v \in \bZ_n 
$$
is an automorphism of $X$ that respects the fundamental edge-partition.
\end{Proposition}

\begin{Remarks}
Note that the mapping $\alpha$ is not uniquely determined 
if there exist distinct $s_1$ and $s_2$ for which $j_{s_1}= \pm j_{s_2}$, 
but we can always define the mapping $\alpha$ so that it is a bijection 
(and satisfies $j_{\alpha(s)} = \pm a j_s$ for all $s \in \bZ_t$). 
Indeed $\alpha$ is uniquely determined if we require that $\alpha(s_1) < \alpha(s_2)$ 
whenever $s_1 < s_2$ and $j_{s_1}= \pm j_{s_2}$. 
%
On the other hand, $\sigma_a$ is not defined when the 
condition $aJ =\{\pm j_0, \pm j_1, \dots, \pm j_{t-1}\}$ fails 
(or equivalently, when $a(J \cup -J) \ne J \cup -J$). 
Note also that $\sigma_1$ is the identity automorphism, while $\sigma_{-1}$ is the 
automorphism $\tau$ defined earlier, since for $a = -1$ we may take $\alpha$ as 
the identity permutation and then $\sigma_{-1}(s,v) = (s,-v) = \tau(s,v)$ for every vertex $(s,v)$. 
\end{Remarks}
\begin{proof}
%
First, let $b$ be the multiplicative inverse of $a$ in $\bZ_n^{\,*}$. 
Then for any $(s,v) \in \bZ_t \times \bZ_n$, we have 
$\sigma_a(\alpha^{-1}(s),bv)=(\alpha(\alpha^{-1}(s)),abv)=(s,v)$, 
and therefore $\sigma_a$ is surjective. 
Since $V(X)$ is finite, it follows that $\sigma_a$ is a permutation. 
Also $\sigma_a$ preserves edges, indeed it respects the  fundamental edge-partition, 
because it takes each neighbour $(s^\prime,v)$ of the vertex $(s,v)$ in the spoke $S_v$ 
to the neighbour $(\alpha(s^\prime), av)$ of the vertex $(\alpha(s), av)$ in the spoke $S_{av}$, 
and takes the two neighbours $(s,v \pm j_s)$ of the vertex $(s,v)$ in the layer $L_s$ 
to the two neighbours $\sigma_a(s,v \pm j_s)=(\alpha(s),a (v \pm j_s))=(\alpha(s), a v \pm a j_s) 
=(\alpha(s),a v \pm j_{\alpha(s)})$ of the vertex $(\alpha(s), av)$ in the layer $L_{\alpha(s)}$. 
\end{proof}

In the remaining part of this section we will show that if the $GI$-graph $X$ is connected, 
then the automorphisms described above and their products give all of the automorphisms of $X$ 
that respect the fundamental edge-partition. 

\smallskip
For this we require two technical Lemmas, the proofs of which are obvious. 

\begin{Lemma}
\label{lemma:spokespoke}
Let $X$ be a connected $GI$-graph with at least two layers. 
Then every automorphism of $X$ that preserves spoke edges 
must permute the spokes $($like blocks of imprimitivity$)$.
\end{Lemma}

\begin{Lemma}            \label{lemma:layers} 
Every automorphism of a $GI$-graph that respects the fundamental edge-partition 
must permute the layer cycles.
\end{Lemma}

It will also be helpful to relate the automorphisms of a $GI$-graph to the automorphisms
of the corresponding circulant graph.

Let $S$ be a subset of $\bZ_n$ such that $S=-S$ and $0 \not \in S$.
Then the \emph{circulant graph} $\Circ(n;S)$ is defined as the graph with 
vertex set $\bZ_n$, such that vertices $u$ and $v$ are adjacent precisely 
when $u-v \equiv a$ mod $n$ for some $a \in S$. 
Equivalently, this is the Cayley graph for $\bZ_n$ given by the subset $S$. 
Note that $\Circ(n;S)$ is connected if and only if $S$ additively generates $\bZ_n$, 
that is, if and only if some linear combination of the members of $S$ is $1$ mod $n$. 

Now suppose that $S=\{s_1, \dots, s_c\}$, and that $\Gamma=\Circ(n;S)$ is connected. 

For $1 \le i \le c$, let $G_{i,1},G_{i,2}, \ldots, G_{i,k_i}$ be the distinct cosets 
of the cyclic subgroup $G_{i,1}=\langle s_i \rangle$ in $G=\langle S\rangle$.
Then we can form a partition $\mathcal C = \{ C_{ij} \}$ 
of the edges of $\Gamma$, where 
$$
C_{ij} =\{\,\{g,g+s_i\}: \, g \in G_{i,j} \,\} \quad \hbox{ for } \, 1 \le j \le k_i \, \hbox{ and } \,1 \le i \le c.
$$ 
Notice that each part $C_{ij}$ of $\mathcal C$ consists of precisely the edges 
of a cycle formed by adding multiples of the single element $s_i$ of $S$ to 
a member of the coset $G_{i,j}$. 

We say that an automorphism $\varphi$ of $\Gamma$ \emph{respects the partition} $\mathcal C$ 
if $\varphi(C_{ij}) \in \mathcal C$ for every  $C_{ij} \in \mathcal C$.  
We have  the following, thanks to Joy Morris: 

\begin{Theorem}  
\label{thm:joymorris}
Suppose the circulant graph $\,\Gamma=\Circ(n;S)$ is connected.  
If $\psi$ is an automorphism of $\Gamma$ which fixes the vertex $0$ and 
respects the partition $\mathcal C = \{ C_{ij} \}$, then $\psi$ is induced 
by some automorphism of $\bZ_{n}$ --- that is, there exists a unit $a \in \bZ_n$ 
with the property that $\, \psi(x)=ax\,$ for every $x \in \bZ_n$ 
$($and in particular, $aS = S)$.
\end{Theorem}

For a proof (by induction on $|S|$), see \cite{Morris}. 
To apply it, we associate with our graph $X= GI(n;J)$ the circulant graph $Y = \Circ(n;S \cup -S)$,  
where $S$ is the underlying set of $J$.  \\ Note that the projection 
$\eta : V(X) \to V(Y)$ given by $\eta(s,v)$ = $v$ takes every layer edge $\{(s,v), (s,v + j_s)\}$ 
of $X$ to the edge $\{v,v + j_s\}$ of $Y$, and hence gives a graph homomorphism from 
the subgraph of $X$ induced on layer edges onto the graph $Y$. 

\begin{Proposition}
\label{propn:XtoY}
Every automorphism of $X= GI(n;J)$ that preserves the set of spoke edges 
induces an automorphism of $Y = \Circ(n;S \cup -S)$ that respects the partition $\mathcal C = \{ C_{ij} \}$. 
\end{Proposition}
\begin{proof}
Any such automorphism $\varphi$ induces a permutation on the set of spokes of $X$, 
and hence under the above projection $\eta$, induces an automorphism of $Y$, say $\psi$. 
Moreover, since $\varphi$ preserves the layer edges, it must permute the layer cycles 
among themselves, and it follows that $\psi$ respects the partition $\mathcal C = \{ C_{ij} \}$. 
\end{proof}

\begin{Corollary}
\label{cor:preserve_layercycles}
Suppose $X$ is connected. Then every automorphism of $X=GI(n;J)$ that 
respects the fundamental edge-partition of $X$ is expressible as a product of powers of 
the rotation $\rho$, the reflection $\tau$, and the automorphisms $\lambda_{i,s_1,s_2}$ 
and $\sigma_a$ defined in Proposition~{\em\ref{propn:mix_layers}} and
Proposition~{\em\ref{propn:change_layers}}.
\end{Corollary}
\begin{proof}
First, any such automorphism $\varphi$ induces a permutation on the set of spokes of $X$, 
and so by multiplying by a suitable element of the dihedral group of order $2n$ 
generated by $\rho$ and $\tau$, we may replace $\varphi$ by an automorphism 
$\varphi'$ that respects the fundamental edge-partition of $X$, and preserves 
the spoke $S_0$.  In particular, $\varphi'$ induces an 
automorphism of $Y = \Circ(n;S \cup -S)$ that fixes the vertex $0$. 
By Theorem~\ref{thm:joymorris}, this automorphism of $Y$ is induced 
by multiplication by some unit $a \in \bZ_n$, and then by multiplying by the 
inverse of $\sigma_a$ we may replace $\varphi'$ by an automorphism 
$\varphi''$ that preserves all of the spokes $S_v$.  
Finally, since $\varphi''$ preserves all of the spokes and also permutes the 
layer cycles among themselves, $\varphi''$ is expressible as a product of the 
automorphisms $\lambda_{i,s_1,s_2}$ defined in Proposition~\ref{propn:mix_layers}. 
\end{proof}

As a special case, we have also the following, for the automorphisms that preserve layers: 

\begin{Corollary} 
\label{cor:preserve_layers}
Suppose $X$ is connected.  Then any automorphism of $X=GI(n;J)$ that 
takes layers to layers is a product of powers of the rotation $\rho$, the reflection $\tau$, 
and the automorphisms $\lambda_{s_1,s_2}$ and $\sigma_a$ defined 
in Corollary~{\em\ref{cor:exchange_layers}} and Proposition~{\em\ref{propn:change_layers}}.
\end{Corollary}

\section{Automorphism groups of $GI$-graphs}
\label{automgps}

Now that we know all possible automorphisms of a $GI$-graph, it is not difficult 
to determine their number, and construct the automorphism groups in many cases. 
We will sometimes use $F(n;J)$ to denote the number of automorphisms of $GI(n;J)$, 
and $A(n;J)$ to denote the automorphism group $GI(n;J)$.

The automorphism group $A(n;J)$ of $GI(n;J)$ always contains a dihedral subgroup 
of order $2n$, generated by the rotation $\rho$ and the reflection $\tau$,
defined in ($\dagger$) in the previous section (before Proposition~\ref{propn:mix_layers}). 
Note that the relations $\rho^n = \tau^2 = (\rho\tau)^2 = 1$ hold, with the third of 
these being equivalent to $\tau\rho\tau = \rho^{-1}$. 

We split the consideration of $F(n;J)$ and $A(n;J)$ into four cases, below. 

\subsection{The disconnected case} 
\label{subs:disconnected}
Let $d = \gcd(n,J)$. Then $GI(n;J)$ is the disjoint union of $d$ isomorphic 
copies of $GI(n;J/d)$. 
This reduces the computation of $\Aut(X)$ to the case of connected $GI$-graphs.  
In particular, 
we have 
\begin{equation*}   \label{eqn:autdis}
A(n;J) \cong A(n,J/d) \wr \Sym (d) \\[+2pt] 
\end{equation*}
so $\Aut(GI(n;J))$ is the wreath product of $\Aut(GI(n;J/d))$ by the symmetric group $\Sym(d)$ of degree $d$, 
and therefore  
\begin{equation*}   \label{eqn:numdis}
F(n,J) = |\!\Aut(GI(n;J))| = d! \,(F(n,J/d))^d.
\end{equation*}

\subsection{The edge-transitive case} 
\label{subs:ET}
The eight connected edge-transitive $GI$-graphs were given in Theorem~\ref{theorem:skewautom}. 
Seven of them are generalized Petersen graphs, with $J = \{1,k\}$ for some $k \in \bZ_n^{\, *}$, 
and their automorphism groups are known --- see \cite{Frucht} or \cite{Lovrecic} for example. 

For each of these seven graphs, all of which are cubic, there is an automorphism $\mu$ of order $3$ 
that fixes the vertex $(0,0)$ and induces a $3$-cycle on it neighbours $(1,0)$, $(0,1)$ and $(0,n-1)$. 
In particular, this automorphism $\mu$ takes the spoke edge $\{(0,0),(1,0)\}$ to the layer edge $\{(0,0),(0,1)\}$, 
and its effect on the other vertices is easily determined. 

In the cases $(n,k) = (4,1)$, $(8,3)$, $(12,5)$ and $(24,5)$, where $n \equiv 0$ mod $4$ 
and $k^2 \equiv 1$ mod $n$, the three automorphisms $\rho$, $\tau$ and $\mu$ generate $A(n;J)$ 
and satisfy the defining relations 
$$
\rho^n = \tau^2 = \mu^3 = (\rho\tau)^2 = (\rho\mu)^2 = (\tau\mu)^2 = [\rho^4,\mu] = 1
$$
for a group of order $12n$ which we may denote for the time being as $\Gamma(n,k)$, 
although strictly speaking, the second parameter $k$ is not necessary. 

Similarly in the case $(n,k) = (10,2)$, the three automorphisms $\rho$, $\tau$ and $\mu$ 
generate $A(n;J)$, which has order $12n$, 
but they satisfy different defining relations, with the relation  $[\rho^4,\mu] = 1$ 
replaced by $\mu\rho^{-1}\mu\rho^{2}\mu^{-1}\rho^{2}\tau = 1$.
In the other two cases (namely $(n,k) = (5,2)$ and $(10,3)$), the 
automorphisms $\rho$, $\tau$ and $\mu$ generate a subgroup of index $2$ in $A(n;J)$, 
which has order $24n$.

In summary, the automorphism groups of the eight connected edge-transitive $GI$-graphs and 
their orders can be described as below: 
\begin{equation*}   \label{eqn:autet}
\begin{array}{rclccrcl}   
\Aut(GI(4,1,1))& \hskip -6pt \cong \hskip -3pt \!&\Gamma(4,1) \ \cong \ S_4 \times \bZ_2 & \ \ & \ \ & F(4,1,1)& \hskip -6pt = \hskip -4pt &48 \\[+2pt] 
\Aut(GI(5,1,2))& \hskip -6pt \cong \hskip -3pt \!&S_5 & \ \ & \ \ & F(5,1,2)& \hskip -6pt = \hskip -4pt &120  \\[+2pt] 
\Aut(GI(8,1,3))& \hskip -6pt \cong \hskip -3pt \!&\Gamma(8,3) & \ \ & \ \ & F(8,1,3)& \hskip -6pt = \hskip -4pt &96 \\[+2pt] 
\Aut(GI(10,1,2))& \hskip -6pt \cong \hskip -3pt \!&A_5 \times \bZ_2 & \ \ & \ \ & F(10,1,2)& \hskip -6pt = \hskip -4pt &120 \\[+2pt] 
\Aut(GI(10,1,3))& \hskip -6pt \cong \hskip -3pt \!&S_5 \times \bZ_2 & \ \ & \ \ & F(10,1,3)& \hskip -6pt = \hskip -4pt &240 \\[+2pt] 
\Aut(GI(12,1,5))& \hskip -6pt \cong \hskip -3pt \!&\Gamma(12,5) & \ \ & \ \ & F(12,1,5)& \hskip -6pt = \hskip -4pt &144 \\[+2pt] 
\Aut(GI(24,1,5))& \hskip -6pt \cong \hskip -3pt \!&\Gamma(24,5)& \ \ & \ \ & F(24,1,5)& \hskip -6pt = \hskip -4pt &288 \\[+2pt] 
\Aut(GI(3,1,1,1))& \hskip -6pt \cong \hskip -3pt \!&(D_6 \times D_6) \rtimes \bZ_2 & \ \ & \ \ & F(3,1,1,1)& \hskip -6pt = \hskip -4pt &72.
\end{array}
\end{equation*}
See \cite{Frucht} and/or \cite{Lovrecic} for further details.

\subsection{The case where $J$ is a set (with no repetitions)}
\label{subs:set}
Suppose $J$ is a set (and not a multiset), in standard form, and let $X = GI(n;J)$. 
If $X$ is not connected, then sub-section~\ref{subs:disconnected} applies, 
while if $X$ is connected and edge-transitive, then sub-section~\ref{subs:ET} applies, 
so we will suppose that $X$ is connected but not edge-transitive.

Then by Corollary  \ref{cor:preserve_layers}, we know that the automorphism group of $X$ 
is generated by the automorphisms $\rho$, $\tau$ and the set $\{\sigma_a : \, a \in A\}$, 
where
$$
A = \{\,a \in Z_n^{\,*}\ | \ a(J\cup-J) = J\cup -J\, \}.
$$
It is easy to see that $A$ is a subgroup of $\bZ_n^{\,*}$. 
Indeed since $\sigma_1$ is trivial, $\sigma_{-1}=\tau$, 
and $\sigma_a \sigma_b=\sigma_{ab}$ for all $a,b \in A$, 
the set $S = \{\sigma_a : \, a \in A\}$ is a subgroup of $\Aut(X)$, isomorphic to $A$. 
In particular, $S$ is abelian. 
It is also easy to see that if composition of functions is read from left to right, 
and $\alpha$ is the bijection satisfying $j_{\alpha(s)} = \pm a j_s$ for all $s \in \bZ_t$, 
then 
$$(\rho \sigma_a)(s,v) = \sigma_a(s,v+1) = (\alpha(s),a(v+1)) = (\alpha(s),av+a) = \rho^a(\alpha(s),av) 
 = (\sigma_a \rho^a)(s,v)$$ 
for every vertex $(s,v)$, and so $\rho \sigma_a = \sigma_a \rho^a$ for all $a \in A$. 
Rearranging, we have $\sigma_a^{-1}\rho \sigma_a = \rho^a$ for all $a \in A$, 
which shows that every element of $S$ normalizes the cyclic subgroup of order $n$ 
generated by the rotation $\rho$.  
Finally, again since $\tau = \sigma_{-1} \in S$, this implies that the automorphism group 
of $X = GI(n;J)$ is a semi-direct product: 
$$
A(n;J) \ = \ \langle\, \{\rho\} \cup S \,\rangle \ \cong \ \langle \rho\rangle \rtimes S \ \cong \ C_n \rtimes A, 
\quad \hbox{ of order } \, F(n;J) = n|A|. 
$$
 
\subsection{The general case}
\label{subs:general}
In this sub-section we deal with all remaining possibilities, in which $J$ is a multiset 
with repeated elements, in standard form, and $X = GI(n;J)$ is connected but not edge-transitive.
Here we need two new sets of parameters, namely the multiplicity $m_j$ in $J$ of each element $j$ 
from the underlying set of $J$ (that is, the number of $s \in \bZ_t$ for which $j_s = j$), 
and $d_j = \gcd(n,j)$ for all such $j$. 

Also we need the set $B$ of all $a \in Z_n^{\,*}$ with the property that 
$aJ =\{\pm j_0, \pm j_1, \dots, \pm j_{t-1}\}$.  
Note that this is always a subgroup of $Z_n^{\,*}$, but is not always the same as the 
subgroup $A = \{\,a \in Z_n^{\,*}\ | \ a(J\cup-J) = J\cup -J\, \}$ 
that we took in the previous sub-section, since the multiplicities of $j$ and $\pm aj$ in $J$ 
might not be the same for some $a \in A$, but clearly they must be the same for every $a \in B$. 

Now by Corollary~\ref{cor:preserve_layercycles} we know that the automorphism 
group of $X$ is generated by the automorphisms $\rho$ and $\tau$, 
the automorphisms $\sigma_a$ for $a \in B$ (as defined 
in Proposition~\ref{propn:change_layers}), and the automorphisms 
$\lambda_{i,s,s^\prime}$ (as defined in Proposition~\ref{propn:mix_layers}) that mix cycles.

Just as in the previous case, the set  $S = \{\sigma_a : \, a \in B\}$ is a subgroup 
of $\Aut(X)$, isomorphic to the subgroup $B$ of $Z_n^{\,*}$. 
Again also we have $\sigma_a^{-1}\rho \sigma_a = \rho^a$ for all $a \in B$, 
and so every element of $S$ normalizes the cyclic subgroup of order $n$ 
generated by the rotation $\rho$.  

Next, for each $j \in J$, define $\Omega_j = \{ s \in \bZ_t \ | \ j_s = j\,\}$, 
which is a set of size $m_j$, and for the time being, let $d = d_j = \gcd(n,j)$. 
Also define  $\Omega_{ji} = \{ (s,v) \in V(X) \ | \ s \in \Omega_j, \ v \equiv i \ {\rm mod} \ d \,\}$, 
for $j \in J$ and $i \in \bZ_d$ (where $d = \gcd(n,j)$). 
Note that $|\Omega_{ji}| = m_{j} n/d$, because $\Omega_{ji}$ is like a strip 
of vertices across $m_j$ layers of $X$, containing the $n/d$ vertices of one cycle 
from each of these layers. 

By Proposition~\ref{propn:mix_layers}, for every two distinct $s_1,s_2$ in $\Omega_j$ 
and every $i \in \bZ_d$, 
there exists an involutory automorphism $\lambda_{i,s_1,s_2}$ that exchanges one of the 
$d$ cycles from layer $L_{s_1}$ with the corresponding cycle from layer $L_{s_2}$, 
and preserves every spoke. 
This automorphism induces a transposition on the set of $m_j$ layer cycles 
containing the vertices of the set $\Omega_{ji}$.  If we let the pair $\{s_1,s_2\}$ vary, 
we get all such transpositions, and hence for fixed $i \in \bZ_d$, the automorphisms 
$\lambda_{i,s_1,s_2}$ with $s_1,s_2$ in $\Omega_j$ generate a subgroup isomorphic 
to the symmetric group $\Sym(m_j)$, acting with $n/d$ orbits of length $m_j$ on $\Omega_{ji}$ 
and fixing all other vertices. 

Moreover, for any two distinct $i_1,i_2$ in $\bZ_d$, the elements of $T_{i_1}$ and $T_{i_2}$ 
move disjoint sets of vertices (namely $\Omega_{ji_1}$ and $\Omega_{ji_2}$), and hence 
commute with each other.  Hence the subgroup $T_j$ generated by all of the automorphisms 
$\lambda_{i,s_1,s_2}$ with $s_1,s_2$ in $\Omega_j$ is isomorphic to the direct 
product of $d$ copies of $\Sym(m_j)$, one for each value of $i$ in $\bZ_d$. 

Similarly, for any two distinct $j,j'$ in $J$, the corresponding subgroups $T_{j}$ 
and $T_{j'}$ move disjoint sets of vertices (from disjoint sets of layers of $X$), 
and hence commute with each other, so the subgroup $N$ generated by the set of 
all of the automorphisms $\lambda_{i,s_1,s_2}$ is a direct product 
$\Pi_{j \in J\,} T_j \cong \Pi_{j \in J\,} (S_{m_j})^{d_{j}}$, of order $\Pi_{j \in J\,} ({m_j}!)^{d_{j}}$. 

On the other hand, for fixed $s_1$ and $s_2$ in $\Omega_j$, then 
$$
\rho^{-1}{\lambda_{i,s_1,s_2\,}}{\rho} = \lambda_{i+1,s_1,s_2} 
\quad \ \hbox{ and } \ \quad 
\tau^{-1}{\lambda_{i,s_1,s_2\,}}{\tau} = \lambda_{-i,s_1,s_2} 
\quad \ \hbox{ for all } \, i \in \bZ_d,  
$$ 
so the automorphisms $\lambda_{i,s_1,s_2}$ are permuted among themselves in 
a cycle under conjugation by the rotation $\rho$, 
and fixed or interchanged in pairs under conjugation by the reflection $\tau$. 

Finally if $a \in B\setminus \{\pm 1\}$, 
and $j'$ is the element of (the underlying set of) $J$ congruent to $\pm aj$ mod $n$,
then the automorphism $\sigma_a$ defined in Proposition~\ref{propn:change_layers} 
takes the layers $L_s$ for $s \in \Omega_j$ to the layers $L_{s'}$ for $s' \in \Omega_{j'}$, 
and conjugates the subgroup $T_j$ (generated by those $\lambda_{i,s_1,s_2}$ 
with $s_1,s_2$ in $\Omega_j$) to the corresponding subgroup $T_{j'}$. 
Hence $\sigma_a$ normalises the subgroup $N = \Pi_{j \in J\,} T_j$. 

Thus $N$ is normalised by $\rho$ and $\tau$ ($=\sigma_{-1}$) and all the other $\sigma_a$, 
and is therefore normal in $\Aut(X)$.  It follows that 
$$
A(n;J) \ = \ \langle\, N \cup \{\rho\} \cup S \,\rangle \ \cong \ N \rtimes \langle \rho\rangle \rtimes S 
\ \cong \ \prod_{j \in J\,} (S_{m_j})^{d_{j}} \rtimes C_n \rtimes B, 
$$
of order \hskip 4cm 
${\displaystyle F(n;J) = n\,|B|\prod_{j \in J} (m_{j}!)^{d_{j}}}, $ \\[+12pt] 
where the products are taken over all $j$ from the underlying set of $J$, 
without multiplicities. 
 
\subsection{Summary}
\label{subs:summary}

Combining the results from the four sub-sections above gives an algorithm 
for computing the automorphism groups of $GI$-graphs and their automorphism 
groups in general.


\section{Vertex-transitive $GI$-graphs}
\label{vertextrans} 

In this section we consider further symmetry properties of $GI$-graphs.
By Corollary \ref{corollary:edgetransitive}, we know there are only eight different 
connected edge-transitive $GI$-graphs $GI(n;J)$ having two or more layers. 
In particular, there are no such graphs with four or more layers. 
In contrast, we will show that there are several vertex-transitive $GI$-graphs, 
by giving a classification of them. 

Note that the graph $GI(n;J)$ will be vertex-transitive if we are able to permute 
the layers of $GI(n;J)$ transitively among themselves. 
Now for each non-zero $a \in \bZ_n$, consider multiplication of the (multi)set $J$ by $a$. 
If this preserves $J$ (as a multiset), then it gives a bijection from $J$ to $J$, and so 
by Proposition~\ref{propn:change_layers}, an automorphism $\sigma_a$ of $GI(n;J)$, permuting the layers. 
The graph $GI(n;J)$ will be vertex-transitive if the group generated by all 
such $\sigma_a$ acts transitively on the layers.

\begin{Theorem}    
\label{thm:vtsubgroup}
Let $J$ be any subset of $\bZ_n^{\,*}$ with the two properties that {\em (a)} $J \cap  -J = \emptyset$,   
and  {\em (b)} $J \cup -J$ is a $($multiplicative$)$ subgroup of $\bZ_n^{\,*}$.
Then $GI(n;J)$ is vertex-transitive. 
\end{Theorem}
\begin{proof}
Since $J \cup -J$ is a subgroup of $\bZ_n^{\,*}$ (not containing $0$), multiplication by any $a \in J$ 
gives a bijection from $J$ to $J$ and hence an automorphism $\sigma_a$ of $GI(n;J)$. 
Moreover, for any $a,b \in J$ there exists $c \in J$ such that $ac=\pm b$ in $\bZ_n$, 
and in this case, the automorphism $\sigma_c$ takes any layer $s$ with $j_s = a$ 
to a layer $s'$ with $j_{s'} = \pm b$. 
It follows that the group generated by $\{\,\sigma_a\! : a\in J \,\}$ acts transitively 
on the layers of $GI(n;J)$, and hence that the group generated by $\{\rho\} \cup \{\,\sigma_a\! : a\in J \,\}$ 
acts transitively on the vertices of $GI(n;J)$. 
\end{proof}

\begin{Corollary}
\label{corollary:VT}
Let $A$ be any subgroup of the multiplicative group $\bZ_n^{\,*}$ 
containing an element of $\bZ_n\setminus \{\pm 1\}$. 
If $-1 \in A$, then take $J = A \cap \{1,2,\dots, \lfloor \frac{n-1}{2} \rfloor \}$ 
$($so that $A = J \cup -J)$, 
while if $-1 \not\in J$, let $J = A$.  
Then $GI(n;J)$ is vertex-transitive. 
Hence for every integer $n > 6$, there exists at least one vertex-transitive $GI$-graph 
of the form $GI(n;J)$ for some $J$ with $|J| > 1$.  
\end{Corollary}

Note that the above requires $\phi(n) = |\bZ_n^{\,*}|$ to be at least $4$, 
so that $n > 4$ and $n \ne 6$, in order for there to be at least two layers. 
A sub-family consists of those for which $A$ is the cyclic 
subgroup $\{1,r,r^2, \ldots, r^{t-1})$ generated by the powers of a single 
unit $r \in \bZ_n^{\,*}\setminus \{\pm 1\}$. 
An example is given in Figure~\ref{fig:twoGIgraphs}(b), with $n = 7$ and $r = 2$ 
(and $2^2 \equiv 4 \equiv -3$ mod $7$). 

\begin{figure}[htb]
\centering
\begin{minipage}[b]{0.45\textwidth}
  \centering
  \includegraphics[width=0.7\textwidth]{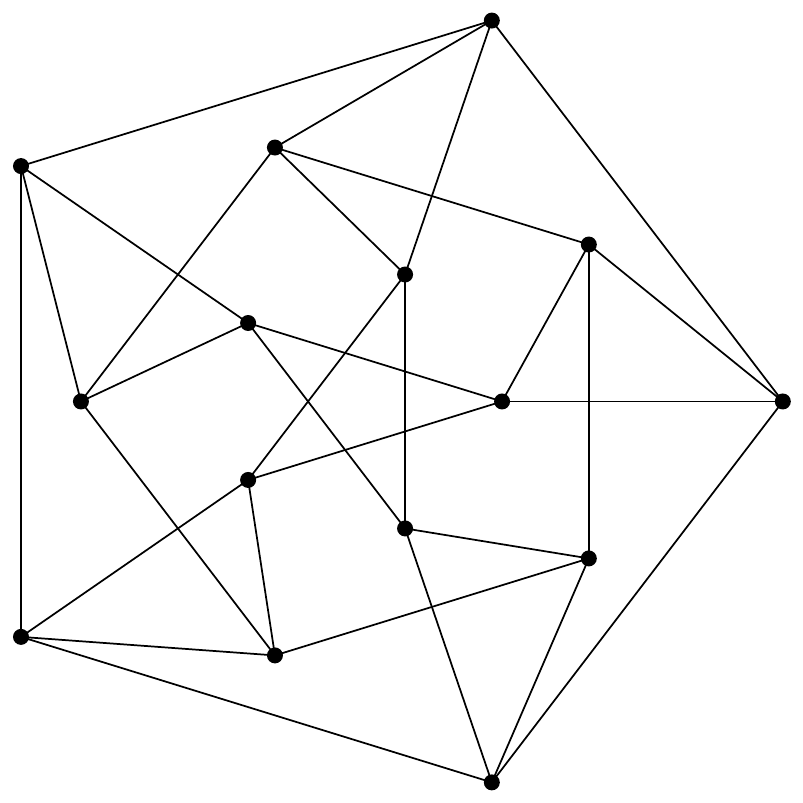}\\ (a)
\end{minipage}
\hfill
\begin{minipage}[b]{0.45\textwidth}
  \centering
  \includegraphics[width=0.7\textwidth]{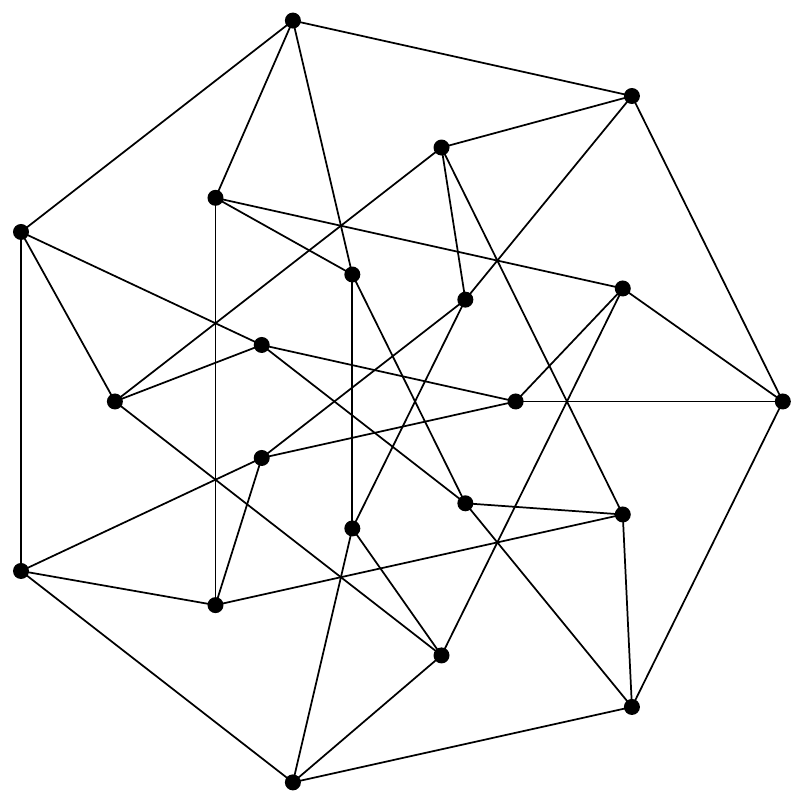}\\ (b)
\end{minipage}
\caption{The graph $GI(5;1,1,2)$ in (a) has 5-cycles as its three layers but is not
vertex-transitive, while the graph $GI(7;1,2,3)$ in (b) is vertex-transitive and has two edge
orbits.} \label{fig:twoGIgraphs}
\end{figure}

Next, we say that a subset  $J=\{j_0,j_1, \dots, j_{t-1}\} $ of $\bZ_n$ is  \emph{primitive} 
if $1 \in J$ and $j_i \ne \pm j_k$ whenever $i \ne k$. 
Also we say that the graph $GI(n;J)$ is \emph{primitive} if $J$ is a primitive subset of $\bZ_n$. 
Note that any such graph is connected, since $1 \in J$.

\begin{Theorem}
\label{thm:primVT}
A primitive $GI$-graph $GI(n;J)$ is vertex-transitive if and only if
either $J \,\cup \,-J$ is a $($multiplicative$)$ subgroup of $\bZ_n^{\,*}$, 
or $n=10$ and $J=\{1,2\}$.
\end{Theorem}
\begin{proof}
First, it was shown in \cite{Frucht} that $GI(10;1,2)$ is vertex-transitive. 
Also by Theorem \ref{thm:vtsubgroup}, we know that $GI(n;J)$ is 
vertex-transitive when $J \cup -J$ is a subgroup of $\bZ_n^{\,*}$. 

Conversely, suppose that $X=GI(n;J)$ is a primitive vertex-transitive 
$GI$-graph, other than $GI(10;1,2)$. 
We have to show that $J \cup -J$ is a subgroup of $\bZ_n^{\,*}$. 

Since $X$ is primitive, we have $1 \in J$, 
and without loss of generality we may assume that $j_0=1$.
By Theorem~\ref{theorem:skewautom}, we know that if $X$ has a skew automorphism, 
then either $t = 2$ and $(n,j_1) = (4,1)$, $(5,2)$, $(8,3)$, $(10, 2)$, $(10, 3)$, $(12, 5)$ 
or $G(24, 5)$, or  $t = 3$ and $(n,j_1,j_2)= (3,1,1)$. 
It is easy to see that $J \cup -J$ is a subgroup of $\bZ_n^{\,*}$ in all of these 
cases except $(n,t,j_0,j_1) = (10,2,1, 2)$.  Hence we may assume 
that $X$ has no skew automorphism, and therefore every automorphism of $X$ 
preserves the fundamental edge-partition.  

Now because $X$ is vertex-transitive, and the layer $L_0$ is a single $n$-cycle, 
it follows that all the layers of $X$ must be cycles, and so every element of $J$ must be 
coprime to $n$, and therefore a unit mod $n$. 
In particular, there are no automorphisms that `mix' cycles from different layers.
In fact, since $J$ is primitive, $J \cup -J$ contains $2t$ distinct elements, 
and it follows that $X$ has no automorphisms of the form given in 
Proposition~\ref{propn:mix_layers} or Corollary~\ref{cor:exchange_layers}. 

Hence (by Theorem~\ref{thm:joymorris}) the only automorphisms that preserve the spoke
$S_0$ are the automorphisms $\sigma_a$ given in Corollary  \ref{cor:preserve_layers}. 

But for any $x \in J$ (say $x = j_s$), by vertex-transitivity there exists an 
automorphism of $X$  that maps the vertex $(0,0)$ to the vertex $(s,0)$, 
and this must be one of the automorphisms $\sigma_a$, where $a$ is a unit in $\bZ_n$ 
and $a(J \cup -J) = J \cup -J$. 
In particular, since $\sigma_a$ takes $(0,v)$ to $(\alpha(0),a v)$ for all $v \in \bZ_n$, 
we have $\alpha(0) = s$ and therefore $x = j_s = j_{\alpha(0)} = \pm a j_0 = \pm a$, 
which gives $x(J \cup -J) = \pm a(J \cup -J) = \pm (J \cup -J) = J \cup -J$, for every $x \in J$.

Thus $J \cup -J$ is closed under multiplication, and by finiteness (and the fact that 
every element of $J$ is a unit mod $n$), it follows that $J \cup -J$ is a subgroup of 
$\bZ_n^{\,*}$, as required. 
\end{proof}


%

\smallskip
We will now find some other examples, and show that every vertex-transitive $GI$-graph has a special form.
To do that, we introduce some more notation: 
we denote by $[k]J$ the concatenation of $k$ copies of the multiset $J$.  
Note that this may involve a non-standard ordering of the elements of $[k]J$, 
but it makes the proofs of some things  in this and the next section easier to explain 
--- specifically, Theorem~\ref{thm:classvt} and Lemmas~\ref{lem:multipleCayley} 
and~\ref{lem:doublenonCayley}. 

\begin{Theorem}                   
\label{thm:classvt}
Let $X=GI(n;J)$ be any connected vertex-transitive $GI$-graph. 
Then 
\\[+2pt] 
{\em (a)} If $\Aut(X)$ has a vertex-transitive subgroup that preserves the 
fundamental edge-partition of $X$, 
then $GI(n;[k]J)$ is vertex-transitive for every positive integer $k$. 
\\[+2pt] 
{\em (b)} All elements in $J \cup -J$ have the same multiplicity, say $k_0$, 
and $($so conversely$)$ the graph $X = GI(n;J)$ is isomorphic to $GI(n;[k_0]J_0)$ 
for some primitive subset $J_0$ of $\bZ_n$, such that $GI(n;J_0)$ is vertex-transitive.
\end{Theorem}
\begin{proof}
Let $X = GI(n;J) = GI(n; j_0,j_1,...,j_{t-1})$, and let $Y = GI(n;[k]J)$.  

Note that the vertex-set of $Y$ is $\bZ_{kt} \times \bZ_n$, 
and we can write $[k]J = (j_0,j_1,...,j_{kt-1})$, where $j_c = j_d$
whenever $c \equiv d$ mod $t$, and accordingly, 
we can write each member $s$ of $Z_{kt}$ in the form $at+b$ 
where $a \in \bZ_k$ and $b \in \bZ_t$. 

Also note that any permutation $f$ of $\{1,2,...,k\}$ gives rise to a corresponding 
permutation $\widetilde{f}$ of $\{1,2,...,kt\}$, defined by setting 
$\widetilde{f}(at+b) = f(a)t+b$ for all $a \in \bZ_k$ and all $b \in \bZ_t$, 
and in fact gives rise to an automorphism $\theta = \theta_f$ of $Y = GI(n;[k]J)$, defined by 
\\[-14pt] 
\begin{center}
$ \theta_f(at+b,v) =  (\widetilde{f}(at+b),v) =  (f(a)t+b,v)
\quad \mbox{for all} \  a \in \bZ_k, \ b \in \bZ_t \, \mbox{ and }\, v \in \bZ_n.$
\\[+6pt] 
\end{center} 
It is easy to see that $\theta_f$ preserves the edges of each spoke $S_v$, 
and permutes the layers among themselves.  
In fact $\theta_f$ takes $L_{at+b}$ to $L_{f(a)t+b}$ for all $a \in \bZ_k$ and all $b \in \bZ_t$, 
and hence $\theta_f$ preserves each of the sets $\{L_s : s \in \bZ_{kt} \, | \, s \equiv b \ {\rm mod} \ t\,\}$ for $b \in \bZ_t$. 

It follows that given any two layers $L_c = \{(c,v) : v \in \bZ_n\}$ and
$L_d = \{(d,v) : v \in \bZ_n\}$ with $c \equiv d$ mod $t$, there exists an
automorphism $\theta$ of $Y$ taking $L_c$ to $L_d$.
In particular, since  $\Aut(Y)$ is transitive on vertices of each layer 
(as is the automorphism group of every $GI$-graph), we find that $\Aut(Y)$ has 
at most $t$ orbits on vertices of $Y$.

%
%

We can now prove (a), by extending certain automorphisms of $X$ to automorphisms 
of $Y$ that make it vertex-transitive. 

Let $\xi$ be any automorphism of $X$ that respects the fundamental edge-partition. 
Define a permutation 
$\pi = \pi_{\xi}$ of the vertex set of $Y$ by letting
$$
\pi(at+b,v) = (at+c,w) \quad \mbox{ whenever } \, \xi(b,v) = (c,w), 
$$
for all $a \in \bZ_k$,  all $b \in \bZ_t$, and all $v \in \bZ_n$.

If $e$ is a spoke edge of $Y$, say from $(at+b,v)$ to $(a't+b',v)$, 
and $(c,w) = \xi(b,v)$, then since $\xi$ takes spoke edges to spoke edges in $X$, 
we see that $\xi(b',v) = (c',w)$ for some $c' \in \bZ_t$, and so by definition  
$\pi(a't+b',v) = (a't+c',w)$, which is a neighbour of $(at+c,w)$. 
Thus $\pi$ takes the edge $e$ to the spoke edge in $Y$ from $(at+c,w)$ to $(a't+c',w)$.
%

Similarly, if $e$ is a layer edge of $Y$, say from $(at+b,v)$ to $(at+b,z)$,
with $z = v+j_b$ (since $j_d = j_{d'}$ whenever $d \equiv d'$ mod $t$), 
and $(c,w) = \xi(b,v)$, then since $\xi$ permutes the layers of $X$, 
we know that $\xi$ takes the neighbour $(b,z) = (b,v+j_b)$
of $(b,v)$ on the same layer of $X$ as $(b,v)$ to a neighbour of $(c,w)$ on the
the same layer of $X$ as $(c,w)$, namely $(c,w \pm j_c)$.
Hence by definition, $\pi(at+b,z) = (at+c,w \pm j_c)$,
which is a neighbour of $(at+c,w)$ in $Y$ because $j_{at+c} = j_c$.
Thus $\pi$ takes $e$ to a layer edge from $(at+c,w)$ 
to $(at+c,w \pm j_c)$ in $Y$.
%

In particular, since $\pi$ preserves both the set of all spoke edges of $Y$ 
and the set of all layer edges of $Y$, we find that $\pi = \pi_{\xi}$ is an automorphism of $Y$. 

Moreover, since $\xi$ can be chosen to take any layer of $X$ to any other layer of $X$, 
it follows that the subgroup of $\Aut(Y)$ generated by the automorphisms $\theta_f$ 
and $\pi_{\xi}$ found above is transitive on layers of $Y$, and hence $Y$ is vertex-transitive.

%
%

\medskip
Next we prove (b), namely that all elements in $J \cup -J$ have the same multiplicity, 
say $k_0$, and $X$ is isomorphic to $GI(n;[k_0]J_0)$ for some 
primitive subset $J_0$ of $\bZ_n$ such that $GI(n;J_0)$ is vertex-transitive.

If $X$ is edge-transitive, then by Theorem~\ref{theorem:skewautom} we have 
$(n;J) = (4;1,1)$, $(5;1,2)$, $(8;1,3)$, $(10;1,2)$, $(10;1,3)$, $(12;1,5)$, $(24;1,5)$  
or $(3;1,1,1)$.  In the first case, we can take $k_0 = 2$ and $J_0 = \{1\}$, and  
observe that $GI(n;J_0) = GI(4,1)$ which is simply a $4$-cycle, and vertex-transitive.  
Similarly, in the last case, we can take $k_0 = 3$ and $J_0 = \{1\}$, and  
observe that $GI(n;J_0) = GI(3,1)$ which is a $3$-cycle, and vertex-transitive.  
In all the other six cases, we can take $k_0 = 1$ and $J_0 = J$, and note that 
$X = GI(n;J)$ itself is vertex-transitive.  Thus (b) holds in all eight cases, and so from 
now on, we may assume that $X$ is not edge-transitive, and hence that 
every automorphism of $X$ respects the fundamental edge-partition. 

This implies that $\Aut(X)$ is transitive on the layers of $X$, and it follows that 
all the layer cycles have the same length, so $\gcd(n,j_s)=\gcd(n,j_0)$ for all $s \in \bZ_t$.
But on the other hand, $X = GI(n;J)$ is connected, so $\gcd(n,j_0,j_1\dots,j_{t-1})=1$. 
Thus $\gcd(n,j_s) = 1$ for all $s \in \bZ_t$.

In particular, there exists $a \in \bZ_n^{\,*}$ such that $1 = aj_s \in aJ$. 
Now by Theorem~\ref{thm:multiply}, the graph $X = GI(n;J)$ is isomorphic to $GI(n;aJ)$, 
and therefore we can replace $J$ by $aJ$, or more simply, suppose that $1 \in J$. 

If all the elements of $J$ are the same, then $X = GI(n;J)$
is isomorphic to $GI(n; [t]\{1\})$, and then since the set $\{1\}$ is primitive
and $GI(n; 1)$ is simply an $n$-cycle, again (b) holds. 

So now suppose that not all elements of $J$ are the same. 
For any two distinct $j_i, j_s \in J$, there must be an automorphism $\sigma_a$ 
that takes layer $L_i$ to layer $L_s$, by Corollary  \ref{cor:preserve_layers}.
In this case  $a(J \cup -J) = J \cup -J$, by definition of $\sigma_a$, 
and therefore the multiplicities of $j_i$ and $j_s$ are the same.
Hence all elements of $J \cup -J$ have the same multiplicity, say $k_0$.          

In particular, $J=[k_0]J_0$ where $J_0$ is the underlying set of $J$, 
and $X$ is isomorphic to $GI(n;[k_0]J_0)$.
The set $J_0$ is primitive since it contains 1 and all of its elements are distinct. 
To finish the proof, all we have to do is show that $GI(n;J_0)$ is vertex-transitive.
But that is easy: for any two distinct $j_i,j_s \in J_0$, we know that 
there exists an automorphism $\sigma_a$ of $X$ taking layer $L_i$ of $X$ to layer $L_s$
of $X$, and $a(J \cup -J) = J \cup -J$; it then follows that $a(J_0 \cup -J_0)=J_0 \cup -J_0$, 
and therefore $\sigma_a$ induces an automorphism of $GI(n;J_0)$ that takes 
layer $L_i$ of $GI(n;J_0)$ to layer $L_s$ of $GI(n;J_0)$, as required.  
\end{proof}

Note that the above theorem above applies only to connected $GI$-graphs.
Disconnected vertex-transitive $GI$-graphs are just disjoint unions of connected 
vertex-transitive $GI$-graphs, and can be dealt with accordingly.

\smallskip
We finish this section with observations about the graphs $GI(5;1,2)$ 
and $GI(10;1,2)$. 

The Petersen graph $GI(5;1,2)$ is vertex-transitive, and its automorphism group 
acts transitively on the two layers;  in fact so does a subgroup of order 20 which  
preserves the set of its ten layer edges. 
By Theorem \ref{thm:classvt}, it follows that every $GI$-graph of the 
form $GI(5;1,2,1,2,\ldots,1,2)$ is vertex-transitive.

On the other hand, the automorphism group of the dodecahedral 
graph $GI(10;1,2)$ has no layer-transitive subgroup preserving the 
set of layer edges (and the set of spoke edges), and so the above theorem 
does not apply to it.  
In fact $GI(10;[k]\{1,2\})$ is not vertex-transitive for any $k > 1$, because 
the fact that $2$ is not a unit mod $10$ implies that the automorphism group 
has two orbits on layers.

The graph $GI(10;1,2)$ is the only such exception, since for every other 
vertex-transitive $GI$-graph $X$, either $\Aut(X)$ itself preserves the 
fundamental edge-partition, or $X$ is edge-transitive and is then 
one of the other seven graphs in Theorem~\ref{theorem:skewautom}, 
and for each of those, the subgroup of $\Aut(X)$ preserving the fundamental 
edge-partition is layer-transitive.

\section{Cayley $GI$-graphs}
\label{cayley} 

In this section we characterise the $GI$-graphs that are Cayley graphs.

First, a Cayley graph ${\rm Cay}(G,S)$ is a graph whose vertices can be labelled 
with the elements of some group $G$, and whose edges correspond to multiplication 
by the elements of some subset $S$ or their inverses.  In particular, the edges of 
${\rm Cay}(G,S)$ may be taken as the pairs $\{g,sg\}$ for all $g \in G$ and all $s \in S$, 
and then the group $G$ acts naturally as a group of automorphisms of ${\rm Cay}(G,S)$ 
by right multiplication.  This action is transitive on vertices, indeed regular on vertices: 
for any ordered pair $(u,v)$ of vertices, there is a unique element of $G$ taking $u$ 
to $v$ (namely $g = u^{-1}v$). 

Alternatively, a Cayley graph is any (regular) graph $X$ whose automorphism group 
has a subgroup $G$ that acts regularly on vertices.  In that case, any particular 
vertex can be labelled with the identity element of $G$, and the subset $S$ can be taken 
as the set of all $s \in G$ taking that vertex to one of its neighbours.  

Note that under both definitions, the Cayley graph is connected if and only if the 
set $S$ generates the group $G$.  Also note that every Cayley graph is vertex-transitive 
(by definition), and that every non-trivial element of the subgroup $G$ fixes no vertices 
of the graph. 

\smallskip
Now suppose $X= GI(n;J)$ is a vertex-transitive $GI$-graph.  

We will assume that $X$ is connected, because if it is not, then it is simply a disjoint 
union of isomorphic copies of a connected smaller example. 
In particular, by Theorem~\ref{thm:classvt}, we know that either $J$ is primitive 
(and $X$ is one of the graphs given by Theorem~\ref{thm:primVT}), 
or all elements in $J \cup -J$ have the same multiplicity $k_0 > 1$ and then $X$ is isomorphic 
to $GI(n;[k_0]J_0)$ for some primitive subset $J_0$ of $\bZ_n$ such that $GI(n;J_0)$ is 
vertex-transitive.

Also we will suppose that $X$ is not $GI(10;1,2)$, for reasons related 
to Theorem~\ref{thm:primVT}. 
In fact, of the seven generalized Petersen graphs among the eight edge-transitive 
$GI$-graphs listed in Theorem~\ref{theorem:skewautom}, it is known by the main 
result of \cite{NedelaSkoviera} or \cite{Lovrecic1} that $G(4,1)$, $G(8,3)$, $G(12, 5)$ 
and $G(24, 5)$ are Cayley graphs, while $G(5,2)$, $G(10, 2)$ and $G(10, 3)$ are not.   
Most of this (and the fact that the eighth edge-transitive graph $GI(3;1,1,1)$ is a Cayley graph) 
will actually follow from what we prove below. 

\smallskip
Consider the case where $J$ is primitive (as we defined in Section~\ref{vertextrans}).
In this case, $J \,\cup \,-J$ is a subgroup of $\bZ_n^{\,*}$ 
under multiplication, and also $|\!\Aut(X)| = n|J \cup -J| = 2n|J| = 2|V(X)|$. 

Hence if $G$ is a subgroup of $\Aut(X)$ that acts regularly on vertices of $X$, 
then $G$ is a subgroup of index $2$ in $\Aut(X)$. 
On the other hand, $G$ cannot contain the element $\tau$, since $\tau$ is 
a non-trivial automorphism with fixed points (namely the vertices $(s,0)$ for all $s$), 
and it follows that $G$ must be generated by the rotation $\rho$ and some 
subgroup of index $2$ in $\{\sigma_a : \, a \in J \cup -J\}$ not containing $\sigma_{-1} = \tau$. 
The latter has to be of the form $\{\sigma_a : \, a \in K\}$ for some subgroup 
$K$ of $J \cup -J$, such that $-1 \notin K$. 

Conversely, if $J$ is a set, and $K$ is a subgroup of index $2$ in $J \cup -J$ 
not containing $-1$, then the group generated by $\{\sigma_a : \, a \in K\}$ permutes 
the layers of $X$ transitively, and so the subgroup generated 
by $\{\rho\} \cup \{\sigma_a : \, a \in K\}$ acts regularly on $V(X)$. 

\smallskip
Thus we have the following:

\begin{Proposition} 
\label{thm:primCayley} 
If $GI(n;J)$ is primitive and $J \cup -J$ is a multiplicative subgroup of $\bZ_n^{\,*},$ 
then $GI(n;J)$ is a Cayley graph if and only if $J \cup -J$ 
has a subgroup of index $2$ that does not contain $-1$. 
\end{Proposition}

Note that this gives infinitely many examples of $GI$-graphs that are Cayley graphs,  
including those where $n$ is a prime congruent to $3$ mod $4$ and $J$ is the 
subgroup of all squares in $\Z_n^{\,*}$. 
On the other hand, it also gives infinitely many vertex-transitive $GI$-graphs that are not Cayley 
graphs, including those where $n$ is a prime congruent to $1$ mod $4$ 
and $J \, \cup\, -J = \Z_n^{\,*} = \Z_n\!\setminus\!\{0\}$.

This theorem also shows that among the six primitive $GI$-graphs that 
are edge-transitive, $G(8;1,3)$, $G(12;1,5)$ and $G(24;1,5)$ are Cayley graphs, 
while $G(5;1,2)$ and $G(10;1,3)$ are not.  
(The graph $GI(10;1,2)$ is not a Cayley graph, for other reasons.)

\smallskip
Next, consider the more general case, where $X=GI(n;J)$ is connected and 
vertex-transitive. 
In this case, by Theorem~\ref{thm:classvt}, 
we know that all elements in $J \cup -J$ have the same multiplicity $k_0$, 
and $X$ is isomorphic to $GI(n;[k_0]J_0)$ for some primitive subset $J_0$ 
of $\bZ_n$, such that $GI(n;J_0)$ is vertex-transitive.
Also by what we found in Section~\ref{vertextrans} and 
sub-section~\ref{subs:general} of Section~\ref{automgps},  
we have $d_j := \gcd(n,j) = 1$ for all $j \in J_0$, and therefore 
$$
|\!\Aut(X)| = n|J_0 \cup -J_0| \prod_{j \in J_0} (k_0!)^{d_j}  
= 2n|J_0| (k_0!)^{|J_0|}. 
$$ 

We will find the following helpful, and to state it, we will refer to the 
automorphism $\rho$ of each $GI$-graph $GI(n;J)$ as its {\em standard rotation}, 
and sometimes denote it by $\rho_J$. 

\begin{Lemma} 
\label{lem:multipleCayley}
If $GI(n;J)$ has a vertex-regular subgroup containing the standard rotation, 
then so does $GI(n;[k]J)$ for every integer $k > 1$.
\end{Lemma}
\begin{proof}
Let $X = GI(n;J)$ and $Y = GI(n;[k]J)$, and let $\rho$ ($= \rho_J$) be the 
standard rotation for $X$.  Also let $\{\rho\} \cup S$ be a generating set  
for a vertex-regular subgroup of $\Aut(X)$. 
Note that $\Aut(X)$ is layer-transitive on $X$, since $X$ is not $GI(10;1,2)$. 
Now by multiplying elements of $S$ 
by powers of $\rho$ if necessary, we may assume that $\langle S \rangle$ 
induces a regular permutation group on the set of layers of $X$.  
In particular, $\langle S\rangle$ has order $J$.
Next, for each $\xi \in S$, the automorphism $\pi_\xi$ defined in the proof 
of Theorem~\ref{thm:classvt} acts fixed-point-freely on $Y = GI(n;[k]J)$, 
and it follows that the set $\{\pi_\xi : \xi \in S\}$ generates a subgroup of order $|J|$ 
that permutes the layers of $Y = GI(n;[k]J)$ in $|J|$ blocks of size $k$. 
Also if $f$ is the $k$-cycle $f = (1,2,\dots,k)$ in $\Sym(k)$, then the 
automorphism $\theta_f$ defined in the proof of Theorem~\ref{thm:classvt} 
induces a $k$-cycle on each of those $|J|$ layer-blocks.  
Finally, $\theta_f$ commutes with $\rho_{[k]J}$ and all the $\pi_\xi$ 
(for $\xi \in S$), so the subgroup generated by $\rho_{[k]J}$, $\theta_f$ 
and all the  $\pi_\xi$ has order $nk|J|$, and acts regularly on the vertices of $Y$, 
as required.  
\end{proof} 

Note that this shows, for example, that both of the remaining  two edge-transitive 
$GI$-graphs $GI(4;1,1)$ and $GI(3;1,1,1)$ are Cayley graphs. 

\smallskip
Somewhat surprisingly, we also have the following:

\begin{Lemma} 
\label{lem:doublenonCayley}
If $J$ is primitive and both $GI(n;J)$ and $GI(n;[2]J)$ are vertex-transitive, then $GI(n;[2]J)$ 
is always a Cayley graph, and so is $GI(n;[k]J)$ for every even integer $k > 1$.
\end{Lemma}
\begin{proof}
First, if $GI(n;J)$ is a Cayley graph, then this follows from Lemma~\ref{lem:multipleCayley}, 
so we will assume that $GI(n;J)$ is not a Cayley graph. 
Also because $GI(n;[2]J)$ is vertex-transitive, we know that $X \ne GI(10;1,2)$, 
and so $J \cup -J$ is a subgroup of $\bZ_n^{\,*}$, by Theorem~\ref{thm:primVT}. 
On the other hand, by Proposition~\ref{thm:primCayley}, we know that $J \cup -J$ has no subgroup 
of index $2$ that excludes $-1$. 
Hence we can write $J \cup -J$ as $U \times W$, where $U$ is a cyclic $2$-subgroup 
containing $-1$ and of order $q = 2^e$ for some $e > 1$, and $W$ is complementary 
to $U$, and of order $2t/q$.  Also let $u$ be a generator of $U$, so that $u^{q/2} = -1$. 

Now consider the automorphisms of $Y = GI(n;[2]J)$.  
For each $a \in J \cup -J = U \times W$, without loss of generality we will choose the 
associated bijection $\alpha: \bZ_{2t} \to \bZ_{2t}$ to be the `duplicate' of the 
corresponding natural bijection from $\bZ_{t}$ to $\bZ_{t}$, namely so that 
$\alpha$ takes $s$ to $s'$, and $s+t$ to $s'+t$, whenever 
$j_{s'} = j_{s'+t} = \pm a j_{s} = \pm a j_{s+t}$ (for $0 \le s < t$). 

For the moment, suppose that $W$ is trivial, so that $U = J \cup -J$. 
Then the automorphism $\sigma_u$ is not semi-regular, 
because the vertex $(0,0)$ lies in a cycle of length $q/2$ consisting of 
all $(s,0)$ with $0 \le s < t$ and $\pm j_{s} = u^i$ for some $i$, while the vertex $(0,1)$ 
lies in a cycle of length $q$ consisting of all $(s,u^i)$ such that $0 \le s < t$ and $\pm j_{s} = u^i$, 
for $0 \le i < q$.
Hence in particular, the subgroup generated by $\rho$ and $\sigma_u$ has order 
$nq = 2nt$, but cannot be vertex-regular (since the $(q/2)$th power of $\sigma_u$ is 
a non-trivial element with fixed points).  

On the other hand, we can multiply $\sigma_u$ by $\lambda_{0,t}$, which 
interchanges vertices $(0,v)$ and $(t,v)$, for all $v \in Z_n$, and find that 
$\sigma_{u}\lambda_{0,t}$ is a semi-regular element of order $q$, with 
$n/q$ cycles of length $q$.  (The vertex $(0,0)$ lies in a cycle of length $q = 2t$ 
consisting of all $(s',0)$ with $\pm j_{s'} = u^i$ for some $i$, while the vertex $(0,1)$ 
lies in a cycle of length $q$ consisting of all $(s,u^i)$ such that $0 \le s < t$ and $\pm j_{s} = u^i$ 
for even $i$, and all $(s+t,u^i)$ such that $0 \le s < t$ and $\pm j_{s} = u^i$ for odd $i$;  
the cycles containing the other vertices $(s',1)$ have a similar form.)

It follows that the subgroup generated by $\rho$ and $\sigma_{u}\lambda_{0,t}$ 
has order $nq = 2nt$, and is transitive on vertices, and hence is vertex-regular, 
so that $GI(n;[2]J)$ is a Cayley graph. 

When $W$ is non-trivial, the elements $\sigma_w$ for all $w$ in $W$ 
(or simply all $w$ from a generating set for $W$) induce a regular permutation 
group on the layers $L_s$ for which $\pm j_s \in W$, and it follows that 
the subgroup generated by $\rho$ and $\sigma_{u}\lambda_{0,t}$ and these 
$\sigma_w$ acts regularly on the vertices of $Y$, again making $GI(n;[2]J)$ 
a Cayley graph.  

Finally, for any even integer $k > 2$, we find that $GI(n;[k]J) = GI(n;[k/2][2]J)$ is 
a Cayley graph, by applying Lemma~\ref{lem:multipleCayley} with $[2]J$ in 
place of $J$, and $k/2$ in place of $k$. 
\end{proof} 

On the other hand, the same does not hold when $k$ is odd:

\begin{Lemma} 
\label{lem:oddnonCayley}
If $J$ is primitive and $GI(n;J)$ is vertex-transitive but 
not a Cayley graph, then $GI(n;[k]J)$ is not a Cayley graph for any odd integer $k > 1$.
\end{Lemma}
\begin{proof}
Assume the contrary, so that $X = GI(n;J)$ is vertex-transitive 
and not a Cayley graph, but $Y = GI(n;[k]J)$ is a Cayley graph, for some odd $k$. 

Then we know that $X \ne GI(10;1,2)$, since $Y$ is vertex-transitive, 
and so $J \cup -J$ is a subgroup of $\bZ_n^{\,*}$, by Theorem~\ref{thm:primVT}. 
On the other hand, since $X$ is not a Cayley graph, Proposition~\ref{thm:primCayley} 
tells us that $J \cup -J$ has no subgroup of index $2$ that 
excludes $-1$,  and therefore $J \cup -J$ contains an element $u$ of (multiplicative) 
order $4m$ for some $m$, with $u^{2m} = -1$. 
Also by Theorem~\ref{theorem:skewautom}, we know that $Y$ is not edge-transitive, 
and so $\Aut(Y)$ preserves the fundamental edge-partition of $Y$, and hence every 
subgroup of $\Aut(Y)$ permutes the layers of $Y$ among themselves. 

Now let $R$ be a vertex-regular subgroup of $\Aut(Y)$, and take $b = u^m$, which has 
order $4$, with $b^2 = -1$ in $\bZ_n^{\,*}$.  
Next, choose $i$ such that $j_i = \pm b$ (noting that such an $i$ must exist because 
$b$ lies in the subgroup $J \cup -J$).  
Then by vertex-transitivity of $R$, there exists some automorphism $\theta$ of $Y$ taking 
the vertex $(0,0)$ to the vertex $(i,0)$. Moreover, by our knowledge of the structure 
of $\Aut(Y)$ from Section~\ref{automgps} and the fact that all of the automorphisms 
$\lambda_{s_1,s_2}$ and $\sigma_a$ preserve the spoke $S_0$, it follows that 
$\theta = w\sigma_b$ or $w\sigma_{-b}$ for some $w$ in the subgroup $N$ generated 
by the set of all of the automorphisms $\lambda_{s_1,s_2}$. 

Since $R$ acts regularly on vertices, every non-trivial automorphism in $R$ has to be 
semi-regular.  In particular, $\theta$ is semi-regular, as is its square 
\\[-14pt] 
\begin{center} 
$\theta^2 = (w\sigma_{\pm b})^2 = w(\sigma_{\pm b}w\sigma_{\pm b}^{\ -1}) \sigma_{(\pm b)^2} 
= w'\sigma_{-1} = w'\tau$,  
\\[-10pt] 
\end{center} 
where $w' = w(\sigma_{\pm b}w\sigma_{\pm b}^{\ -1}) \in N$.
Both $w'$ and $\tau$ preserve the spoke $S_0$, and therefore so does $w'\tau$, 
and thus $w'\tau$ acts semi-regularly on $S_0$.  But also $\tau$ fixes every 
vertex $(s,0)$ of $S_0$, and so $w'$ itself acts semi-regularly on $S_0$.  
Furthermore, since every element of $N$ preserves the set $\{L_0,L_t,\dots,L_{(k-1)t}\}$ 
of $k$ layers corresponding to the occurrences of $1$ in  $J$, 
it follows that both $w'\tau$ and $w'$ act semi-regularly on the set $K = \{(rt,0) : 0 \le r < k\}$. 

In particular, cycles of the permutation induced by $w'$ on $K = \{(rt,0) : 0 \le r < k\}$ 
must all have the same length, say $\ell$.  Note that $w'$ is non-trivial, for otherwise 
$w'\tau = \tau$, which is not semi-regular on vertices (because it has fixed points), 
and therefore $\ell > 1$.  But also $\ell$ must divide $k$, so $\ell$ is odd. 

Now consider any $\ell$-cycle of $w'$ on $K$,  say $((s_1,0),(s_2,0), \dots, (s_\ell,0))$. 
Because $\tau$ fixes every vertex of $K$, this is also a cycle of $w'\tau$, 
and hence all cycles of $w'\tau$ have length $\ell$. 
Also by definition of the elements generating $N$ 
(as defined in Proposition~\ref{propn:mix_layers}), 
we know that $((s_1,1),(s_1,1), \dots, (s_\ell,1))$ must be a cycle of $w'$. 
But now the cycle of $w'\tau$ containing the vertex $(s_1,1)$ is 
\\[-14pt] 
\begin{center} 
$((s_1,1),(s_2,-1),(s_3,1), \dots,(s_{\ell-1},-1),(s_{\ell},1),
(s_1,-1),(s_2,1),(s_3,-1), \dots,(s_{\ell},-1)),$
\\[-10pt] 
\end{center} 
which has length $2k$, and this contradicts the fact that $w'\tau$ is semi-regular.
\end{proof} 


Putting together Proposition~\ref{thm:primCayley} and Lemmas~\ref{lem:multipleCayley} 
and~\ref{lem:doublenonCayley}, we have the following: 

\begin{Theorem} 
\label{thm:nonprimCayley} 
If $X=GI(n;J)$ is connected, then $X$ is a Cayley graph if and only if 
\\[+2pt] 
{\em (a)} $J$ is primitive, and $J \cup -J$ is a multiplicative subgroup of $\bZ_n^{\,*},$ 
with a subgroup of index $2$ that does not contain $-1$,  or 
\\[+2pt] 
{\em (b)} 
$X = GI(n;J)$ is isomorphic to $GI(n;[k_0]J_0)$ 
for some primitive subset $J_0$ of $\bZ_n$ and some integer $k_0 > 1$, 
such that either $GI(n;J_0)$ is a Cayley graph,  
or $k_0$ is even and $GI(n;J_0)$ is vertex-transitive but is not the dodecahedral graph $GI(10;1,2)$. 
\end{Theorem}

\section{Additional remarks}
\label{conclusion} 

The family of $GI$-graphs forms a natural generalisation of the Petersen graph.
Our initial studies of $GI$-graphs have shown that this family
is indeed very interesting  and deserves further consideration.
These graphs are also related to circulant graphs \cite{Morris}. Through that 
relationship, we were able to solve the puzzle of what appeared to be 
unstructured automorphisms of $GI$-graphs, and this enabled us to find their
automorphism groups and classify those that are vertex-transitive or Cayley graphs.

Let us mention also the problem of unit-distance drawings of $GI$-graphs.
A graph is a \emph{unit-distance graph} if it can be drawn in the plane 
such that  all of its edges have the same length.
In \cite{HPZ1}, it was shown that all $I$-graphs are unit-distance graphs.
On the other hand, obviously no $GI$-graph with four or more layers can be a unit-distance
graph, since it contains a $K_4$ as a subgraph, which itself is not a unit-distance 
graph. Hence the only open case of interest is the sub-class of $GI$-graphs 
having three layers.

For each $k \in \Z_n$, the graph $GI(n; k,k,k)$ is a cartesian product of two cycles 
and is therefore a unit-distance graph by \cite[Theorem 3.4]{HorvatPisanski2}. 
We know of only one other connected example that is a unit-distance graph, 
and it is remarkable.  

This is the graph $GI(7;1,2,3)$, which is shown 
in Figure \ref {fig:UDGIgraphs}. The vertices can be drawn
equidistantly on three concentric circles with radii 
$$R_1=\frac{1}{2\sin(\pi/7)}, \ \ \ 
  R_2=\frac{1}{2\sin(2\pi/7)}, \ \ \ \mbox{and} \ \ \
  R_3=\frac{1}{2\sin(3\pi/7)},
$$  
and the two smaller circles rotated through angles of $\pi/3$ and
$-\pi/3$  with respect to the largest circle.
One can then verify that all edges have the same length $1$. 

\begin{figure}[htb]
\centering
  \includegraphics[width=0.35\textwidth]{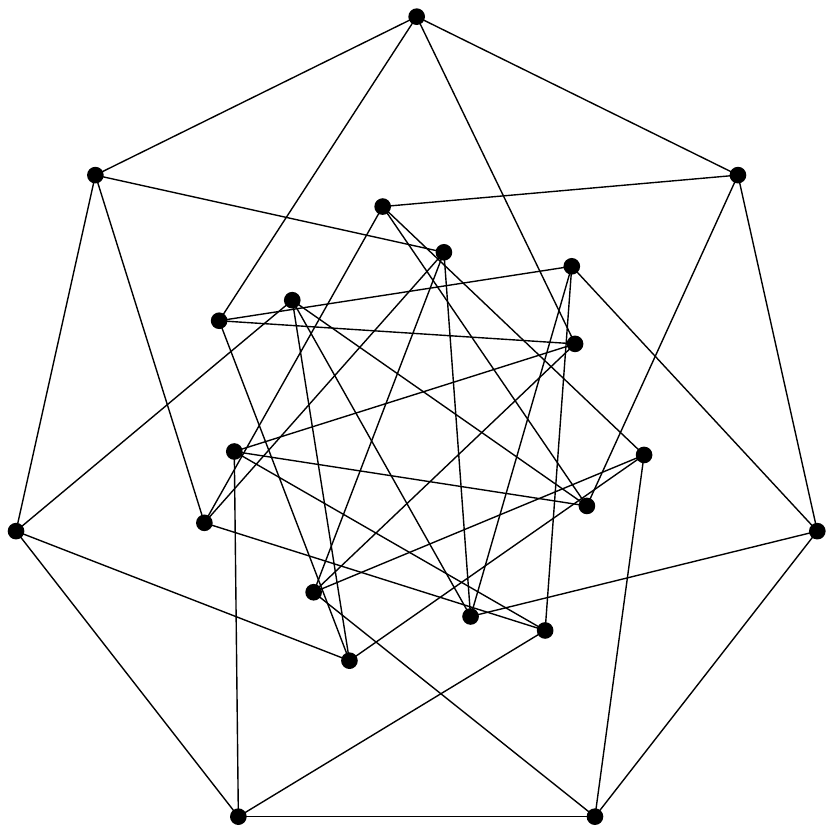}
\caption{The graph  $GI(7;1,2,3)$ as a unit-distance graph.} 
\label{fig:UDGIgraphs}
\end{figure}

The graph $GI(7;1,2,3)$ is a Cayley graph for the non-abelian group 
of order $21$, namely $\bZ_7 \rtimes_{2\,} \bZ_3$, which has presentation 
$\langle \, a,b \, | \ a^7 = b^3 = 1, \, b^{-1}ab = a^2 \, \rangle$. 
Its girth is 3 but it contains no cycles of length 4. 
This means that its Kronecker cover (see \cite{ImrichPisanski}) has girth 6 and is
a Levi graph \cite{Coxeter} of a self-polar, point- and line-transitive but not 
flag-transitive combinatorial $(21_4)$-configuration.
The resulting configuration is different from the configuration of Gr\"unbaum 
and Rigby \cite{GR}, since the latter configuration is flag-transitive but 
the one obtained from $GI(7;1,2,3)$ is not.

\section*{Acknowledgements}
This work has been financed by ARRS 
within the EUROCORES Programme EUROGIGA (project GReGAS, N1--0011) 
of the European Science Foundation. 
It was also supported partially by the ARRS (via grant P1-0294), 
and the N.Z.~Marsden Fund (via grant UOA1015).



\end{document}